\documentclass{amsart}
\usepackage{amssymb}
\usepackage{amsfonts}
\usepackage{amsmath}
\usepackage{array}
\usepackage{url}

\newtheorem{theorem}{Theorem}[section]
\newtheorem{lemma}[theorem]{Lemma}
\newtheorem{thm}{Theorem}
\theoremstyle{definition}
\newtheorem{definition}[theorem]{Definition}
\newtheorem{example}[thm]{Example}

\theoremstyle{remark}
\newtheorem{remark}[theorem]{Remark}

\numberwithin{equation}{section} \theoremstyle{plain}

\newtheorem{corollary}[theorem]{Corollary}

\newtheorem{proposition}[theorem]{Proposition}

\copyrightinfo{2001}{enter name of copyright holder}

\newcommand{\df}{\mathrm{d}}

\newcommand{\Sph}{\mathbb{S}}
\input{tcilatex}

\begin{document}
\title[Warped product Einstein manifolds]{On the classification of warped product Einstein
metrics}
\author{Chenxu He}
\address{209 S 33rd St\\
Dept. of Math, University of Pennsylvania\\
Philadelphia, PA 19104.}
\curraddr{14 E. Packer Ave\\
Dept. of Math, Lehigh University \\
Christmas-Saucon Hall\\
Bethlehem, PA, 18015}
\email{he.chenxu@lehigh.edu}
\urladdr{\url{http://sites.google.com/site/hechenxu/}}
\author{Peter Petersen}
\address{520 Portola Plaza\\
Dept of Math UCLA\\
Los Angeles, CA 90095}
\email{petersen@math.ucla.edu}
\urladdr{\url{http://www.math.ucla.edu/~petersen}}
\thanks{The second author was supported in part by NSF-DMS grant 1006677}
\author{William Wylie}
\address{209 S 33rd St \\
Dept. of Math, University of Pennsylvania\\
Philadelphia, PA 19104.}
\email{wylie@math.upenn.edu}
\urladdr{\url{http://www.math.upenn.edu/~wylie}}
\thanks{The third author was supported in part by NSF-DMS grant 0905527 }
\date{}

\begin{abstract}
In this paper we take the perspective introduced by Case-Shu-Wei of
studying warped product Einstein metrics through the equation for
the Ricci curvature of the base space. They call this equation on
the base the $m$-Quasi Einstein equation, but we will also call it
the $(\lambda,n+m)$-Einstein equation. In this paper we extend the
work of Case-Shu-Wei and some earlier work of Kim-Kim to allow the
base to have non-empty boundary. This is a natural case to consider
since a manifold without boundary often occurs as a warped product
over a manifold with boundary, and in this case we get some interesting
new canonical examples. We also derive some new formulas involving
curvatures which are analogous to those for the gradient Ricci solitons.
As an application, we characterize warped product Einstein metrics
when the base is locally conformally flat.
\end{abstract}

\maketitle

\section{Introduction}

In this paper a $\left( \lambda ,n+m\right) $-Einstein manifold $(M^n,g,w)$
is a complete Riemannian manifold, possibly with boundary, and a smooth
function $w $ on $M$ which satisfies
\begin{eqnarray}
\mathrm{Hess} w &=& \frac{w}{m} ( \mathrm{Ric} - \lambda g)  \notag \\
w &>&0 \text{ on int}(M)  \label{eqnlambdaEinstein} \\
w &=& 0 \text{ on } \partial M.  \notag
\end{eqnarray}
When $m=1$ we make the additional assumption that $\Delta w = -\lambda w$.

These metrics are also called \emph{$m$-Quasi Einstein} metrics in \cite{CSW}. Our motivation for studying this equation is that, when $m$ is an integer, $(\lambda, n+m)$-Einstein metrics are exactly those $n$-dimensional manifolds which are the base of an $n+m$ dimensional Einstein warped product.

\begin{proposition}
\label{propwarpedproduct} \label{warping} Let $m> 1$ be an integer. Then
there is a smooth $(n+m)$ dimensional warped product Einstein metric $g_E$
of the form
\begin{eqnarray*}
g_E = g + w^2 g_{F^m}
\end{eqnarray*}
if and only if $(M,g,w)$ is a $\left (\lambda, n+m\right)$-Einstein manifold.
\end{proposition}

\begin{remark}
The $m=1$ case is discussed in \cite{Corvino} and the case where $M$ has no
boundary was proved in \cite{KimKim}.
\end{remark}

A simple example of a $(\lambda, n+m)$-Einstein metric is when $w$ is constant. Then $\mathrm{Ric} = \lambda g$ and $\partial M = \emptyset$, and we call the space a $\lambda$-Einstein manifold. Note that a $\lambda$-Einstein manifold is $(\lambda, n+m)$- Einstein for all $m\geq 1$ and the
warped product is a Riemannian product. In this case we say the space is a \emph{trivial} $(\lambda, n+m)$-Einstein manifold.

By varying the parameter $m$, the $(\lambda, n+m)$-Einstein equation formally interpolates between two well known equations. $(\lambda, n+1)$-Einstein metrics are more commonly called \emph{static metrics}. Static metrics have been studied extensively for their connections to scalar
curvature, the positive mass theorem, and general relativity. See for example \cite{Anderson}, \cite{Corvino}, \cite{AndersonKhuri1}, \cite{AndersonKhuri2}. Note that, from the relativity perspective, it is natural to assume in addition that the metric is asymptotically flat see e.g \cite{Israel}, \cite{Robinson}, \cite{BuntingMasood}, however we will not consider that assumption in this paper. The more general question of Einstein metrics that are submersion metrics on a fixed bundle over a fixed Riemannian manifold has also been studied extensively in the physics
literature where they are called Kaluza-Klein metrics. The equations can be found in Chapter 9 of \cite{Besse}, also see \cite{Betounes}, \cite{CGS}, and \cite{OverduinWesson}. For many other related results about warped product metrics also see \cite{DobarroUenal2010} and the references therein.

On the other hand, if we define $f$ in the interior of $M$ by $e^{-f/m} = w$, then the $(\lambda, n+m)$-Einstein equation becomes
\begin{eqnarray*}
\mathrm{Ric}_f^m = \mathrm{Ric} + \mathrm{Hess } f - \frac{\mathrm{d}
f\otimes \mathrm{d} f}{m} = \lambda g,
\end{eqnarray*}
$\mathrm{Ric}_f^m$ is sometimes called the \emph{$m$-Bakry Emery tensor}. Lower bounds on this tensor are related to various comparison theorems for the measure $e^{-f} \df \mathrm{vol}_g$, see for example Part II of \cite{Villani}, and \cite{Bayle}, \cite{Morgan}, \cite{WeiWylie}. From these
comparison theorems, the $(\lambda, n+m)$-Einstein equation is the natural Einstein condition of having constant $m$-Bakry-Emery Ricci tensor. Taking $m \longrightarrow \infty$, one also obtains the gradient Ricci soliton equation
\begin{eqnarray*}
\mathrm{Ric} + \mathrm{Hess }f = \lambda g.
\end{eqnarray*}
We could then also call a gradient Ricci soliton a $(\lambda, \infty)$-Einstein manifold. Ricci solitons have also been studied because of their connection to Ricci flow (See for example \cite{Chow}, Chapter 1).

In this paper we develop some new equations for $(\lambda, n+m)$-Einstein
metrics which are analogous to formulas for gradient Ricci solitons.
Interestingly, this analogy with gradient Ricci solitons works very
well when $m>1$, but seems to break down when $m=1$. For a different
connection between Ricci solitons and static metrics see \cite{AkbarWoolgar}.

As mentioned above, when $\partial M = \emptyset$, $(\lambda, n+m)$-Einstein metrics have been studied in \cite{CSW} and \cite{KimKim}. In this paper these results will be extended to the case where $\partial M \neq \emptyset$. For example, we have the following extension of Theorem 1 in \cite{KimKim}.

\begin{theorem}
\label{thmKimKim} \label{compact} If $(M,g,w)$ is a compact $(\lambda, n+m)$-Einstein manifold with $\lambda \leq 0$ then it is trivial.
\end{theorem}

\begin{remark}
Another interesting result from \cite{CSW} is a classification of K\"{a}hler $(\lambda, n+m)$-Einstein metrics for finite $m$. Their arguments also give a classification when $\partial M \neq \emptyset$, see Remark \ref{remKah} below. Other extensions of results from \cite{CSW} are discussed in section 5.
\end{remark}

The main part of this paper is to introduce new formulas for $(\lambda, n+m)$-Einstein metrics. As an application, we obtain a classification of complete locally conformally flat $(\lambda, n+m)$-Einstein manifolds with $m>1$. The proof is motivated by the corresponding result for gradient Ricci solitons proven independently in \cite{CaoChen} and \cite{CatinoMantegazza}.

\begin{theorem}
\label{thmharmonic-Weyl} Let $m>1$ and suppose that $(M,g)$ is complete,
simply connected, and has harmonic Weyl tensor and $W(\nabla w,\cdot ,\cdot
,\nabla w)=0$, then $(M,g,w)$ is a non-trivial $(\lambda, n+m)$-Einstein
metric if and only if it is of the form
\begin{eqnarray*}
g &=& \mathrm{d}t^2 + \psi^2(t)g_L \\
w &=& w(t),
\end{eqnarray*}
where $g_L$ is an Einstein metric. Moreover, if $\lambda \geq 0$ then $(L,g_{L})$ has non-negative Ricci curvature, and if it is Ricci flat, then $\psi$ is a constant, i.e., $(M,g)$ is a Riemannian product.
\end{theorem}

\begin{remark}
This global result is obtained by applying a similar local characterization. If one removes the complete and simply connected assumption, we still get a characterization of $(M,g,w)$ around certain points in $M$, see Theorem \ref{ThmLocal}. If one only assumes completeness, then we can apply the theorem to the universal cover. As quotients of the universal cover must preserve $w$, the possible non-simply connected spaces are quite restricted. One advantage of having the arguments be local is that it is then also possible to obtain a classification in the more general case where $M$ is a Riemannian orbifold. In fact it follows that, unless $\psi$ is positive everywhere, than the orbifold is a finite quotient of a $(\lambda,n+m)$-Einstein metric on the
sphere, disc or Euclidean space.
\end{remark}

\begin{remark}
Even when $\lambda >0$ there are interesting examples of rotationally
symmetric $(\lambda, n+m)$-Einstein metrics, see \cite{Bohm1}. This is in
sharp contrast to the gradient Ricci soliton case where there are no
unexpected examples \cite{Kotschwar}.
\end{remark}

\begin{remark}
In dimension three harmonic Weyl tensor is equivalent to local conformal flatness. This assumption cannot be weakened when $n=3$ as there are local examples which are not warped products (see Example \ref{exa:3DnonCF}). However, when $n>3$, harmonic Weyl tensor is a weaker condition
since it is equivalent to divergence free Weyl tensor while local conformal flatness is equivalent to vanishing Weyl tensor. For simple examples of $(\lambda ,n+m)$-Einstein metrics which have divergence free Weyl tensor and zero radial Weyl tensor but are not locally conformally flat consider the examples listed in table \ref{TablekappaEinstein} where N and F are Einstein
metrics which do not have constant sectional curvature. Such examples exist for $n\geq 5$.
\end{remark}

\begin{remark}
When $n=1$ and $2$ a classification of $(\lambda ,n+m)$-Einstein manifolds can be found in \cite{Besse}, we will discuss these examples in section 3 and Appendix A. See \cite{Seshadri} for the case where $n=3$, $m=1$, and the metrics have symmetry.
\end{remark}

\begin{remark}
The local characterization of locally conformally flat $(\lambda, n+m)$
Einstein metrics has also been obtained independently by Catino, Mantegazza,
Mazzieri, and Rimoldi in a recent paper \cite{Catinoetc}. They posted their
preprint on the arxiv in October 2010, as we were finalizing this
manuscript. Both their work and ours has been motivated by the corresponding
results for gradient Ricci solitons. Whereas they give a different proof of
the result that also gives a new and interesting approach in the Ricci
soliton case, our approach was to set up a system of formulas on $(\lambda,
n+m)$-Einstein metrics which are analogous to the formulas for gradient
Ricci solitons. These calculations use two tensors $Q$ and $P$ which we
introduce in section 6. We will give other applications of these
calculations in \cite{HPWvirtual} and \cite{HPWrigidity}.
\end{remark}

In this paper we also modify Case-Shu-Wei's classification of $(\lambda,
n+m) $-Einstein manifolds which are also $\kappa$-Einstein for $\kappa \neq
\lambda$, see Proposition \ref{propkappaEinstein}. In \cite{HPWvirtual} we
consider solutions to the $(\lambda, n+m)$-Einstein equation where $w$ can
change sign. In \cite{HPWrigidity} we consider $(\lambda, n+m)$-Einstein
manifolds with constant scalar curvature and classify such manifolds under
certain additional curvature conditions.

\smallskip

The paper is organized as follows. In section 2, we study the properties of
the boundary and prove Proposition \ref{propwarpedproduct}. In section 3, we
discuss some examples including the classification in lower dimensions. In
sections 4 and 5 we discuss the modifications of the results in \cite{KimKim}
and \cite{CSW} to allow $\partial M \neq \emptyset.$ In section 6, we
develop some new formulas for $(\lambda, n+m)$-Einstein manifolds. In
section 7, we apply these formulas to prove Theorem \ref{thmharmonic-Weyl}.
Finally, in the appendix we have also included a sketch of the classification
on $(\lambda, m+2)$-Einstein manifolds which is also outlined in \cite{Besse}.

\medskip

\textbf{Acknowledgment:} The authors would like to thank Esther Cabezas
Rivas, Huai-Dong Cao, Jeffrey Case, Dan Knopf, Guofang Wei, and Wolfgang
Ziller for enlightening conversations and helpful suggestions which helped
us with our work.

\section{Properties of the boundary and the warping construction}

In this section we collect some simple facts about the behavior of $w$ near
the boundary of a $\left( \lambda ,m+n\right) $-Einstein manifold. When $m=1$
all of these results can be found in section 2 in \cite{Corvino}. The proofs
when $m>1$ are similar but we include them for completeness. We then apply
these facts about $\partial M$ to prove Proposition \ref{propwarpedproduct}.
Throughout this section we will let $(M,g,w)$ be a non-trivial $(\lambda,
n+m)$-Einstein manifold, i.e., $w$ is not a constant function.

The following formula is proven by Kim-Kim using a local calculation
involving the Bianchi identities.

\begin{proposition}
\cite[Proposition 5]{KimKim} \label{propmu} There is a constant $\mu $ such
that
\begin{equation*}
\mu =w\Delta w+(m-1)|\nabla w|^{2}+\lambda w^{2}.
\end{equation*}
\end{proposition}

\begin{remark}
The constant $\mu$ is the Ricci curvature of the fiber $F$ of the warped
product Einstein metric over $M$. When $m=1$, the extra condition $\Delta w
= - \lambda w$ is equivalent to $\mu=0$, which is necessary for the
existence of a one dimensional $F$.
\end{remark}

\begin{remark}
By tracing the $(\lambda, n+m)$-Einstein equation, we have
\begin{equation*}
\Delta w = \frac{w}{m} (\mathrm{scal} - n \lambda),
\end{equation*}
and then we can re-write the equation for $\mu$ as
\begin{eqnarray}
\mu &=& k w^2 + (m-1) |\nabla w|^2  \label{eqnmu} \\
k &=& \frac{ \mathrm{scal}+(m-n)\lambda}{m}.  \notag
\end{eqnarray}
\end{remark}

\begin{remark}
There is a similar identity on gradient Ricci soltions,
\begin{equation*}
\mathrm{scal}+|\nabla f|^{2}-2\lambda f=\mbox{const}.
\end{equation*}
\end{remark}

The first fact about $\partial M$ we are after is the following proposition.

\begin{proposition}
\label{regval} $|\nabla w| \neq 0$ on $\partial M$.
\end{proposition}

\begin{proof}
Let $x_0 \in \partial M$ and let $\gamma(t)$ be a unit speed geodesic emanating from $x_0$ such that $\gamma^{\prime}(0) \perp \partial M$. Let $h(t) = w(\gamma(t))$ and $\Theta(t) = \mathrm{Ric}(\gamma^{\prime}(t), \gamma^{\prime}(t)) - \lambda$. Then the equation for $w$ becomes a linear second order ODE along $\gamma$ for $h$:
\begin{eqnarray*}
h^{\prime\prime}(t) &=& \mathrm{Hess} w (\gamma^{\prime}(t),
\gamma^{\prime}(t)) \\
&=& \frac{1}{m} \Theta(t) h(t). \\
h(0) &=& 0 \\
h^{\prime}(0) &=& g( \nabla w, \gamma^{\prime})_{x_0}.
\end{eqnarray*}
Therefore, If $\nabla w(x_0) = 0$, then $h^{\prime}(0 )=0$ and so $h = 0$
along all of $\gamma$. Since $\gamma(t) \in \mathrm{int}(M)$ for $0 < t<
\varepsilon$, this is a contradiction.
\end{proof}

This also gives us the following.

\begin{proposition}
\label{propboundary} The boundary $\partial M$ is totally geodesic and $|\nabla w|$ is constant on the connected components of $\partial M$.
\end{proposition}

\begin{proof}
The equation
\begin{equation*}
\mathrm{Hess}w=\frac{w}{m}(\mathrm{Ric}-\lambda g)
\end{equation*}
shows that $\mathrm{Hess}w=0$ on $\partial M.$ This shows that $\partial M$ is totally geodesic as the second fundamental form is proportional to $\left( \mathrm{Hess}w\right) |_{\partial M}.$ It also shows that $|\nabla w|^2$ is locally constant along $\partial M$ since
\begin{eqnarray*}
D_{X} |\nabla w|^2 = \mathrm{Hess }w(X, \nabla w).
\end{eqnarray*}
\end{proof}

When $m>1$ we now get something slightly different than in the $m=1$ case.

\begin{corollary}
\label{cormuboundary} If $m>1$ then $|\nabla w|^2$ is globally constant on $\partial M$ and, moreover, if $\partial M \neq \emptyset$ then $\mu > 0$.
\end{corollary}

\begin{proof}
On $\partial M$, (\ref{eqnmu}) becomes
\begin{equation}  \label{boundarymu}
\mu =(m-1)|\nabla w|^{2}
\end{equation}
so $\mu>0$ and $|\nabla w|^2$ is determined by $m$ and $\mu$.
\end{proof}

\smallskip

Now we prove Proposition \ref{propwarpedproduct}.

\begin{proof}[Proof of Proposition \protect\ref{propwarpedproduct}]
Let $g_{E}$ be the warped product Riemannian metric
\begin{eqnarray*}
g_{E} &=&g+w^{2}g_{F} \\
\mathrm{Ric}_{g_{F}} &=&\mu g_{F},
\end{eqnarray*}
where $F$ is an $m$-dimensional Einstein metric. The calculations in either \cite{Besse} or \cite{KimKim} show that $\mathrm{Ric}_{g_{E}}=\lambda g_{E}$. If $\partial M=\emptyset$ we then have that $g_{E}$ is smooth Einstein metric on the topological product $M\times F$. If $\partial M\neq \emptyset$, then $g_{E}$ is a metric on
\begin{equation*}
E=(M\times F)/\sim
\end{equation*}
where $(x,p)\sim (x,p^{\prime })$ if $x\in \partial M$. Note that near $\partial M,$ the topology of the space is $\partial M\times F^{m}.$ Note that Corollary \ref{cormuboundary} implies that $\mu >0$ when $m>1$ so we can take $F=\mathbb{S}^{m}$, which is necessary for $E$ to be a smooth
manifold. We only have to show that $g_{E}$ is a smooth metric on $E$, normalize so that $\mu =m-1$, then we have
\begin{equation*}
kw^{2}+(m-1)|\nabla w|^{2}=m-1.
\end{equation*}
The above equation shows that $\left\vert \nabla w\right\vert =1$ on $\partial M$ and this guarantees that we obtain a smooth metric. To see this write $g=\frac{1}{\left\vert \nabla w\right\vert ^{2}}\mathrm{d}w^{2}+g_{w}$ near $\partial M$ so that
\begin{eqnarray*}
g_{E} &=&\frac{\mathrm{d}w^{2}}{\left\vert \nabla w\right\vert ^{2}}+g_{w}+w^{2}g_{F} \\
&=&\mathrm{d}w^{2}+g_{w}+w^{2}g_{F}+O\left( 1-\left\vert \nabla w\right\vert
^{2}\right) \mathrm{d}w^{2}.
\end{eqnarray*}
Here $\mathrm{d}w^{2}+g_{w}+w^{2}g_{F}$ defines a smooth metric and the last term $O\left( 1-\left\vert \nabla w\right\vert ^{2}\right) \mathrm{d}w^{2}$ vanishes at $\partial M$ showing $g_{E}$ defines a smooth metric.
\end{proof}

Finally we show that the metric as well as $w$ have to be real analytic.

\begin{proposition}
Let $\left( M,g,w\right) $ be $\left( \lambda ,m+n\right) $-Einstein. Then $g $ and $w$ are real analytic in harmonic coordinates on $\mathrm{int}M.$
\end{proposition}

\begin{proof}
We proceed as in \cite[Theorem 5.26]{Besse} using the two equations
\begin{eqnarray*}
\frac{w}{m}\mathrm{Ric}-\mathrm{Hess}w-\frac{w}{m}\lambda g &=&0, \\
w\Delta w+\left( m-1\right) \left\vert \nabla w\right\vert ^{2}+\lambda
w^{2}-\mu &=&0
\end{eqnarray*}
In harmonic coordinates this looks like the quasi-linear system
\begin{eqnarray*}
-\frac{w}{2m}\sum g^{rs}\frac{\partial ^{2}g_{ij}}{\partial x^{r}\partial x^{s}}-\frac{\partial ^{2}w}{\partial x^{i}\partial x^{j}}+\text{\textrm{lower order terms}} &=&0 \\
-w\sum g^{rs}\frac{\partial ^{2}w}{\partial x^{r}\partial x^{s}}+\text{\textrm{lower order terms}} &=&0
\end{eqnarray*}
We note this is elliptic as long as $w>0.$ In addition the whole system is of the form $F\left(g, w, \partial g,\partial w,\partial ^{2}g,\partial ^{2}w\right) =0$ where $F$ is real analytic. The claim then follows.
\end{proof}

\section{Examples}

In this section we review the classification of one and two dimensional $(\lambda, n+m)$-Einstein metrics which are stated in \cite{Besse}. We also give the characterization of non-trivial $(\lambda, n+m)$-Einstein manifolds which are also Einstein. We will often reduce to this characterization when proving the later results. Finally we reference some interesting examples in
higher dimensions that can be found in the literature.

We begin with the simplest case, the one dimensional examples.

\begin{example}[One dimensional examples -- See 9.109 in \protect\cite{Besse}]
\label{exampleoneD}
Suppose $M$ is one dimensional, then the $(\lambda, 1+m)$-Einstein equation is
\begin{eqnarray*}
w^{\prime\prime}= - k w , \quad \mbox{where} \quad k = \frac{\lambda}{m}.
\end{eqnarray*}
Then, up to re-parametrization of $t$, $w$ must be one of the following examples, here the metric is $g = \mathrm{d} t^2$ and $C$ is an arbitrary positive constant.

\begin{table}[!h]
\begin{center}
\vspace{.2 in} {\ \offinterlineskip
\tabskip=0pt
\halign{
\vrule height2.75ex depth1.25ex width 0.6pt #\tabskip=1em &
\hfil # \hfill &\vrule #& \hfill  #  \hfill &\vrule#& \hfill  # \hfill &\vrule #& \hfill # \hfill \vrule width 0.6pt \tabskip=0pt\cr
\noalign{\hrule height 0.6pt}
&  &&$\lambda >0 $ && $\lambda=0$&& $\lambda <0 $ \cr
\noalign{\hrule}
& $\mu>0$ && $M=[-\frac{\pi}{2\sqrt{k}}, \frac{\pi}{2\sqrt{k}}]$ &&$ M= [0, \infty) $ && $M=[0, \infty)$  \cr
& && $w(t) = C \mathrm{cos(\sqrt{k}t)}$ && $w(t) = Ct$ && $w(t) = C\sinh(\sqrt{-k}t)$  \cr
\noalign{\hrule}
& $\mu=0$ && None  &&$ M=\mathbb{R} $ && $M=\mathbb{R}$  \cr
& &&   && $w(t) = C$ && $w(t) = Ce^{\sqrt{-k}t}$  \cr
\noalign{\hrule}
& $\mu <0 $ && None  && None  && $M= \mathbb{R}$  \cr
& &&   &&  && $w(t) = C \cosh(\sqrt{-k}t) $  \cr
\noalign{\hrule height 0.6pt}
}}
\par
\vspace{.2 in}
\end{center}
\caption{One dimensional $(\protect\lambda, m+n)$-Einstein manifolds.}
\label{table1D}
\end{table}
\end{example}

\begin{remark}
When we construct the warped product metrics $g_E$ over the 1-dimensional examples we obtain various ways to write the constant curvature spaces as warped products over one-dimensional bases. The entry in the middle of Table \ref{table1D} is the trivial solution and the other five are non-trivial.
\end{remark}

More simple examples arise from classifying $(\lambda, n+m)$-Einstein manifolds which are also $\rho$-Einstein for some constant $\rho \neq \lambda $. This is an extension of Proposition 4.2 in \cite{CSW} to manifolds with boundary.

\begin{proposition}
\label{propkappaEinstein} Suppose that $(M^n,g,w)$ with $n \geq 2$ is a non-trivial $(\lambda, m+n)$-Einstein manifold which is also $\rho$-Einstein, then it is isometric to one of the examples in Table \ref{TablekappaEinstein}, where $\bar{\kappa} = \frac{\lambda - \rho}{m}$.
\end{proposition}

\begin{table}[!h]
\begin{center}
\vspace{.2 in}
\par
{\ \offinterlineskip
\tabskip=0pt
\halign{
\vrule height2.75ex depth1.25ex width 0.6pt #\tabskip=0.6em &
\hfil # \hfill &\vrule #& \hfill  #  \hfill &\vrule#& \hfill  # \hfill &\vrule #& \hfill # \hfill \vrule width 0.6pt \tabskip=0pt\cr
\noalign{\hrule height 0.6pt}
&  &&$\lambda>0 $ && $\lambda=0$&& $\lambda<0 $ \cr
\noalign{\hrule}
& $\mu > 0 $ && $\mathbb{D}^{n}$ &&$ [0, \infty) \times F$ &&   $[0, \infty) \times N$  \cr
& && $g = \df t^{2}+ \sqrt{\bar{k}} \sin ^{2}(\sqrt{\bar{k}} t) g_{\Sph^{n-1}} $ && $g =\df t^{2}+ g_{F} $ && $g =\df t^{2}+ \sqrt{-\bar{k}} \cosh ^{2}(\sqrt{-\bar{k}} t) g_{N}$  \cr
& && $w(t) = C \cos(\sqrt{\bar{k}}t)$ && $w(t) = C t$ && $w(t) = C \sinh(\sqrt{-\bar{k}}t)$  \cr
\noalign{\hrule}
& $\mu=0$ && None  && None && $(-\infty , \infty ) \times F$  \cr
& &&   &&   && $g = \df t^{2}+e^{2\sqrt{-\bar{k}}t}g_{F} $  \cr
& &&   && && $w(t) = C e^{\sqrt{-\bar{k}}t}$  \cr
\noalign{\hrule}
& $\mu < 0 $ && None  && None  && $\mathbb{H}^n$  \cr
& &&   &&  && $g = \df t^2 +  \sqrt{-\bar{k}} \sinh^2(\sqrt{-\bar{k}} t) g_{\Sph^{n-1}}$  \cr
& &&   &&  && $w(t) = C\cosh(\sqrt{-\bar{k}}t) $  \cr
\noalign{\hrule height 0.6pt}
}}
\par
\vspace{.2 in}
\par
\smallskip
\end{center}
\caption{Non-trivial $(\protect\lambda,m+n)$-Einstein manifolds that are also Einstein. When $m=1$, $\protect\mu=0$. When $m>1$, the sign of $\protect \mu$ is given in the left hand side of the table. Here $\mathbb{S}^{n-1}$ is a round sphere, $F$ is Ricci flat, $N$ is an Einstein metric with negative Ricci curvature, and $C$ is an arbitrary positive constant}
\label{TablekappaEinstein}
\end{table}

\begin{proof}
Suppose that $\mathrm{Ric} = \rho g$ and let $\bar{k}=\frac{\lambda -\rho }{m}$. Then we have
\begin{equation*}
\mathrm{Hess}w=-\bar{k}wg.
\end{equation*}

If $\bar{k}=0$ then we have $\mathrm{Hess} w = 0$, so if $w$ is non-constant then it must be a multiple of a distance function and the metric must split along $w$, so we obtain the product metric in the $\lambda = 0$, $\mu>0$ entry in the table.

On the other hand, if $\bar{k} \neq 0$, then $w$ is a strictly convex or concave function on the interior of $M$ and therefore it can have at most one isolated critical point. Now the fact that $\mathcal{L}_{\nabla w}g=2\mathrm{Hess}w=-2\bar{k}wg$, tells us that $w = w(t)$ and we have
\begin{eqnarray*}
g &=& \mathrm{d} t^{2}+(w^{\prime }(t))^{2}g_{S} \\
w^{\prime \prime }(t) &=&-\bar{k}w.
\end{eqnarray*}
where $t$ is the distance to the critical point (if it exists), or is the distance to a level set if $\nabla w$ never vanishes. When there is a critical point at $t=0$, $g_S$ must be the round sphere to obtain a smooth metric, if there is no critical point, then $g_S$ is the metric of a level set of $w$. The result follows easily from these equations. For more details see, \cite{Brinkmann}, \cite{Besse}, \cite{PWclassification}, or \cite{CheegerColding}.
\end{proof}

\begin{remark}
Note that when $\lambda$ and $\mu$ have the same sign, then the space is
either trivial or is the simply connected space of constant curvature. In
the other examples, if we let $N$ or $F$ have constant curvature, then one
also gets constant curvature spaces, but if $N$ or $F$ do not have constant
curvature we get examples which are Einstein but do not have constant
curvature.
\end{remark}

\begin{remark}
In both of the tables above there are empty spaces when $\lambda \geq 0$ and
$\mu\leq 0$. It turns out that there are no non-trivial examples in these
cases in general. See Corollary \ref{TableUTriangular} in the next section.
\end{remark}

\begin{remark}
It is also important to notice here that the Einstein constant, $\rho $ of
the metrics in Table \ref{TablekappaEinstein} is $\rho =(n-1)\bar{k}$. Since
$\bar{k}=\frac{\lambda -\rho }{m}$, this shows that $\rho $ is determined by
$\lambda ,n,$ and $m$ via the formula
\begin{equation}\label{rholambda}
(m+n-1)\rho =(n-1)\lambda .
\end{equation}
\end{remark}

\begin{remark}
It is also interesting to consider the behavior of the examples in Table \ref{TablekappaEinstein} as $m \rightarrow \infty$. That is, fix $\lambda$ and $n $ and let $m \rightarrow \infty$. From (\ref{rholambda}), we see that $\rho \rightarrow 0$ and so $\bar{k} \rightarrow 0$ as well. We first consider the $\lambda > 0, \mu>0$ case. In this case we can see that the diameter of the hemispheres are expanding and the metric is becoming flat. The Riemannian measure on $g_E$ is
\begin{eqnarray*}
\df \mathrm{vol}_{E} = w^m \df \mathrm{vol}_M \otimes \df \mathrm{vol}_F
\end{eqnarray*}

So we have the natural measure on $M$, $w^m \df \mathrm{vol}_g$ associated to
the $(\lambda, n+m)$-Einstein metric. In our case this becomes,
\begin{eqnarray*}
\lim_{m\rightarrow \infty} w^m(t) &=&\lim_{m\rightarrow \infty} \cos ^m (
\sqrt{\bar{k}} t ) \\
& = & \lim_{m\rightarrow \infty}\left( 1 - \frac{\lambda - \rho}{m} \frac{
t^2}{2} + \cdots \right)^m \\
&=&\lim_{m\rightarrow \infty} \left( 1 - \frac{\lambda t^2}{2m} \right)^m \\
&=& e^{-\frac{\lambda}{2} t^2}.
\end{eqnarray*}
Therefore we can see that this family of examples converge to the shrinking Gaussian on $\mathbb{R}^n$. For this reason one could call the $(\lambda, n+m)$-metric on the hemisphere the \emph{elliptic Gaussian}. By the same argument we can also see that the constant curvature $(\lambda, n+m)$-Einstein metrics with $\lambda <0$ and $\mu<0$ will converge to the expanding Gaussian $e^{-\frac{\lambda}{2} t^2}$ with $\lambda<0$.

On the other hand, in the cases where $\lambda$ and $\mu$ do not have the same sign, there is no convergence. For example, if $w = e^{\sqrt{-\bar{k}}t} $, then
\begin{equation*}
w^m = e^{\sqrt{ m(\rho - \lambda)}t}
\end{equation*}
so the function $w^m$ degenerates as $m \rightarrow \infty$, going to zero if $t<0$, staying constant at $t=0$, and blowing up if $t>0$. The other examples behave similarly.
\end{remark}

We now turn our attention to the classification of surface $(\lambda ,2+m)$-Einstein metrics, which is stated in \cite{Besse}. First we have the various examples with constant curvature from Table \ref{TablekappaEinstein}. It is straight forward to see from the analysis in \cite{Besse}(also see
Appendix A) that these are the only examples with $m=1$ or $\partial M\neq \emptyset$. Theorem 1.2 in \cite{CSW} also shows that there are no non-trivial compact examples with $\partial M\neq \emptyset$. When $M$ is non-compact, there are a few examples with non-constant curvature and we
discuss some interesting ones.

\begin{example}[Generalized Schwarzschild metric]
\label{exampleschwarzschild}

Let $w$ be the unique positive solution on $[0,\infty )$ to
\begin{equation*}
(w^{\prime })^{2}=1-w^{1-m}\quad \mbox{with}\quad w(0)=1,\quad w^{\prime
}\geq 0.
\end{equation*}
Then
\begin{eqnarray*}
g &=& \mathrm{d}t^2 + (w^{\prime}(t))^2 \mathrm{d} \theta ^2 \\
w &=&w(t)
\end{eqnarray*}
is the unique $(0,2+m)$-Einstein metric with non-constant curvature(see \cite{Besse}, Example 9.118(a)). This is a rotationally symmetric metric on $\mathbb{R}^2$ with $\mu>0$. If we set $w=r$ we can write the metric $g_{E}$ as
\begin{eqnarray*}
g_E = \frac{1}{1-r^{1-m}} \mathrm{d} r^2 + \left( 1 - r^{1-m} \right)
\mathrm{d} u^2 + r^2 g_{\mathbb{S}^m}
\end{eqnarray*}
which, when $m=2$, is the usual way to write the Schwarzschild metric on $E = \mathbb{R}^2 \times \mathbb{S}^2$.

In \cite{Case2} it is shown that as $m \rightarrow \infty$ these metrics converge to Hamilton's cigar, which is the unique rotationally symmetric steady gradient Ricci soliton, so one could also call calls these metrics $m$-cigars.
\end{example}

\begin{remark}
From the viewpoint of general relativity it is more natural to view the
Schwarzschild as a static metric. To see this just reverse the roles of $w$
and $w^{\prime}$, then
\begin{eqnarray*}
\overline{M} &=& [0, \infty) \times \mathbb{S}^{n-1} \\
\bar{g} &=& \mathrm{d} t^2 + w^2(t) g_{\mathbb{S}^{n-1}} \\
\bar {w}& = &w^{\prime}
\end{eqnarray*}
is a static metric which also has the generalized Schwarchild metric as its
total space $g_E$.
\end{remark}

\begin{example}
When $\lambda < 0$ and $m>1$ there are also two additional families of examples of $(\lambda, 2+m)$-Einstein metrics which do not have constant curvature. The first are translation invariant metrics in an axis (see \cite{Besse} 9.118 (c)). These examples all must have $\mu<0$. If we quotient these examples in the axis of symmetry we also obtain examples on the cylinder $\mathbb{R} \times \mathbb{S}^1$. On the other hand, Example 9.118 (d) in \cite{Besse} gives rotationally symmetric examples with $\lambda <0$. These examples can have $\mu$ positive, zero, or negative.
\end{example}

\smallskip

\begin{remark}
\label{remKah} As we mention in the introduction, the classification of K\"{a}hler $(\lambda, n+m)$-Einstein metrics in \cite{CSW} goes through to the case where we allow boundary. In particular, the arguments in Theorem 1.3 in \cite{CSW} show that the universal cover must split as the product of a $\lambda$-Einstein metric and a two dimensional solution. In particular
this implies that a compact K\"{a}hler $(\lambda, n+m)$-Einstein metric must either be trivial or the product of a $\lambda$-Einstein metric with the elliptic Gaussian.
\end{remark}

In higher dimensions there are other interesting constructions of $(\lambda, n+m)$-Einstein metrics.

\begin{example}
B\"{o}hm \cite{Bohm1} has constructed non-trivial rotationally symmetric $(\lambda, n+m)$-Einstein metrics on $\mathbb{S}^n$ and $\mathbb{D}^n$ for $n=3, 4, 5, 6, 7.$

Examples with $\lambda \leq 0$ are also constructed in \cite{Bohm2}. It is also proven that for each $m$, there is a unique rotationally symmetric $(0,n+m)$-Einstein metric on $\mathbb{R}^n$. The other examples are not locally conformally flat.

Other examples are constructed by L\"{u}, Page, and Pope in \cite{LvPagePope}. For $m \geq 2$ they construct non-trivial cohomogeneity one $(\lambda, n+m)$-Einstein metrics on some $\mathbb{S}^2$ and $\mathbb{R}^2$-bundles over K\"{a}hler Einstein metrics. These examples have $\mu>0$ and the $\mathbb{S}^2 $-bundles have $\lambda>0$ while the $\mathbb{R}^2$ bundles have $\lambda =0$. The lowest dimension for this construction is four and it is the one on the nontrivial $\mathbb{S}^2$-bundle over $\mathbb{CP}^1$, i.e. $\mathbb{CP}^2 \sharp \overline{\mathbb{CP} ^2}$ where $\overline{\mathbb{CP}^2}$ has the opposite orientation. These examples are also not locally conformally flat.
\end{example}

Finally we show that in dimension 3 there are local solutions to the $\left(\lambda,m+n\right)$-Einstein equations which are not locally conformally flat.
\begin{example}
\label{exa:3DnonCF}
Consider a doubly warped product metric
\begin{equation*}
g = \mathrm{d}r^{2}+\phi^{2}\mathrm{d}\theta_{1}^{2}+\psi^{2}\mathrm{d}\theta_{2}^{2}
\end{equation*}
then the equations
\begin{equation*}
\frac{m}{w}\mathrm{Hess}w = \mathrm{Ric}-\lambda I
\end{equation*}
become
\begin{eqnarray*}
m\frac{w^{\prime\prime}}{w} & = & -\frac{\phi^{\prime\prime}}{\phi}-\frac{\psi^{\prime\prime}}{\psi}-\lambda\\
m\frac{w^{\prime}\phi^{\prime}}{w\phi} & = & -\frac{\phi^{\prime\prime}}{\phi}-\frac{\phi^{\prime}\psi^{\prime}}{\phi\psi}-\lambda\\
m\frac{w^{\prime}\psi^{\prime}}{w\psi} & = & -\frac{\psi^{\prime\prime}}{\psi}-\frac{\phi^{\prime}\psi^{\prime}}{\phi\psi}-\lambda
\end{eqnarray*}
These can be solved near $r=0$ using suitable initial conditions. For example when $\phi(0)=\psi(0)>0$, $\phi^{\prime}(0)\ne\psi^{\prime}(0)$, $w(0)>0$, and $w^{\prime}(0)=0$ we obtain a local solution to the $(\lambda,3+m)$-equations that is not locally conformally flat.
\end{example}

\medskip

\section{Compact $(\protect\lambda, m+n)$ Einstein spaces}

In this section we discuss the proof Theorem \ref{compact} and some other
general facts about compact $(\lambda ,m+n)$ Einstein manifolds.

We begin with some notation. For $a \in \mathbb{R}$ we consider the measure $\df \mu_a = w^{a}\df \mathrm{vol}_{g}$ on $M.$ There is a naturally defined Laplacian associated to the measure
\begin{eqnarray*}
L_a(u) &=&w^{-a}\mathrm{div}\left( w^{a}\nabla u\right) \\
&=&\Delta u+w^{-a}g\left( \nabla u,\nabla w^{a}\right) \\
&=&\Delta u+aw^{-1}g\left( \nabla u,\nabla w\right),
\end{eqnarray*}
which is self-adjoint when $M$ is closed. When $M$ has boundary we also have that $w=0$ on $\partial M.$ This means that $\df \mu_a$ is finite as long as $a>-1.$ Note that $L_a(u)$ is only defined on $\mathrm{int}(M)$ and can blow at the boundary unless $u$ satisfies the weighted condition:
\begin{equation*}
\left\vert \nabla u\right\vert \leq Cw.
\end{equation*}
In any case, the divergence theorem gives us the following lemma.

\begin{lemma}
\label{IntByParts}On a compact $(\lambda, m+n)$-Einstein manifold, if $u, v$
are functions such that
\begin{equation}  \label{boundaryterm}
\lim_{x \rightarrow \partial M}v w^a g(\nabla u, \nabla w) = 0
\end{equation}
then
\begin{equation*}
\int_{M} v L_a(u) \df \mu_a =-\int_{M}g\left( \nabla v,\nabla u\right) \df\mu_a.
\end{equation*}
\end{lemma}

\begin{proof}
Let $U_{\varepsilon} = \{ x \in M: w(x) \geq \varepsilon \}$. By the
divergence theorem,
\begin{equation*}
\int_{U_{\varepsilon}}vL_a(u) \df \mu_a=-\int_{U_{\varepsilon}}g\left( \nabla
v, \nabla u\right) \df \mu_a -\int_{\partial U_{\varepsilon} }g\left( \frac{\nabla w}{\left\vert \nabla w\right\vert },v\nabla u\right) \df \mu_{a |\partial U_{\varepsilon}}.
\end{equation*}
The result follows since $\nabla w \neq 0$ on $\partial M$.
\end{proof}

\begin{remark}
If $a>0$ and $|\nabla u|$ and $v$ don't blow up at $\partial M$ then the
condition (\ref{boundaryterm}) is always satisfied, but we will also apply
this lemma when $a<0$.
\end{remark}

We encounter two sets of formulas involving $L_a$. The first such example is the following re-interpretation of Kim-Kim's identity as the formula for $L_{m-2}(w^2)$, in the next section we see that $a=m+1$ arises when considering the scalar curvature.

\begin{proposition}
On a $(\lambda, n+m)$-Einstein manifold,
\begin{equation*}
L_{m-2}(w^{2})=-2\lambda w^{2}+2\mu
\end{equation*}
\end{proposition}

\begin{proof}
We have
\begin{eqnarray*}
\nabla(w^2) &=& 2 w \nabla w \\
\mathrm{Hess} (w^2) &=& 2 \mathrm{d} w \otimes \mathrm{d} w + 2 w \mathrm{Hess} w \\
\Delta(w^2) &=& 2 |\nabla w|^2 + 2 w \Delta w.
\end{eqnarray*}
So the equation
\begin{eqnarray*}
\mu = w\Delta w + (m-1)|\nabla w|^2 + \lambda w^2
\end{eqnarray*}
can be re-written as
\begin{eqnarray*}
-2\lambda w^{2}+2\mu &=&2w\Delta w+2(m-1)|\nabla w|^{2} \\
&=&\Delta (w^{2})+2(m-2)|\nabla w|^{2} \\
&=&L_{m-2} (w^{2}).
\end{eqnarray*}
\end{proof}

\begin{corollary}
Define the function
\begin{equation}  \label{eqnphi}
\phi :=\left\{
\begin{array}{cc}
w^{2}-\frac{\mu }{\lambda } & \text{ if }\lambda \neq 0, \\
w^{2} & \text{ if }\lambda =0.
\end{array}
\right.
\end{equation}
Then, if $M$ is compact,
\begin{equation*}
L_{m-2}(\phi )=-2\lambda \phi
\end{equation*}
on the interior of $M$.
\end{corollary}

\begin{proof}
In the $\lambda \neq 0$ case this is just a way of re-writing the previous proposition. When $\lambda =0$, the previous proposition tells us that
\begin{equation*}
L_{m-2}(\phi )=2\mu.
\end{equation*}
We want to show $\mu =0$. In the case where the boundary is empty, $w$ must have an interior maximum and minimum, which implies $\mu =0$. On the other hand, if $\partial M\neq \emptyset$ by Corollary \ref{cormuboundary} we know that $\mu \geq 0$. Moreover $w$ is a non-negative function which is zero on the boundary and so must have in interior maximum, which implies $\mu \leq 0$.
\end{proof}

\begin{remark}
We have only used compactness in the $\lambda =0$ case, so the formula is true in general for $\lambda \neq 0$. If $\lambda =0$ we only have $L_{m-2}(w)=2\mu $ in general. There are examples with $\lambda =0 $ and $\mu >0$, so compactness is necessary in this case.
\end{remark}

We can now apply Lemma \ref{IntByParts} to prove the extension of Kim-Kim's theorem to manifolds with boundary.

\begin{proof}[Proof of Theorem \protect\ref{compact}]
We wish to apply Lemma \ref{IntByParts} to $u=v = \phi$ and $a=m-2$. In order to do so we must check that $w^{m-2}\phi \nabla \phi $ goes to zero at the boundary. From the definition of $\phi$, (\ref{eqnphi}), we have
\begin{equation*}
w^{m-2}\phi \nabla \phi =2w^{m-1}\phi \nabla w.
\end{equation*}
The right hand side goes to zero if $m>1$. On the other hand, if $m=1,$ then $\mu =0$ so $\phi=0$ on $\partial M$, and so the quantity also goes to zero in this case. Then by Lemma \ref{IntByParts}
\begin{eqnarray*}
\int_{M}|\nabla \phi |^{2}\df \mu_{m-2} &=&-\int_{M}L_{m-2}(\phi)\phi \df \mu_{m-2}
\\
&=&2\lambda \int_{M}\phi ^{2}\df \mu_{m-2}
\end{eqnarray*}
so $\lambda >0.$
\end{proof}

In fact, we point out that the converse of Kim-Kim's theorem is also true.

\begin{theorem}
A non-trivial $\left( \lambda ,m+n\right) $-Einstein manifold is compact if
and only if $\lambda > 0$.
\end{theorem}

\begin{proof}
When $m$ is an integer this is a consequence of Myers' theorem applied to
the warped product metric on $E$. When $m$ is not an integer, we can prove
this by applying an extension of Myers' theorem due to Qian \cite{Qian}.
Since the boundary is totally geodesic, we see that the minimal geodesic
between two interior points of $M$ is completely contained in the interior
of $M$. In the interior we also have
\begin{eqnarray*}
\mathrm{Ric}_f^m = \lambda g > 0.
\end{eqnarray*}
The arguments in \cite{Qian} then show that the length of the geodesic is
uniformly bounded above, and so the manifold is compact.
\end{proof}

Using these formulas we can also prove that $\mu>0$ when $\lambda>0$ and $m>1 $.

\begin{proposition}
\label{propmuintegral} \label{propcompactmu} Suppose that $m>1$ and $\lambda
>0$, then
\begin{equation*}
\mu =\lambda \frac{\int_{M}\df \mu_{m}}{\int_{M}\df \mu_{m-2}}.
\end{equation*}
In particular $\mu >0$.
\end{proposition}

\begin{proof}
We have
\begin{equation*}
w^{m-2}\nabla \phi =2w^{m-1}\nabla w
\end{equation*}
which goes to zero at $\partial M$ when $m>1$, so Lemma \ref{IntByParts}
implies
\begin{equation*}
-2\lambda \int_{M}\phi \df\mu_{m-2}=\int_{M}L_{m-2}(\phi )\df\mu_{m-2}=0.
\end{equation*}
From the definition of $\phi$ (\ref{eqnphi}), we have the formula for $\mu$.
\end{proof}

This is also the final step in obtaining the following fact which was referred to in the previous section. The proof appeals to the result in the $\partial M = \emptyset$ case in \cite{Case1}.

\begin{corollary}
\label{TableUTriangular} The only $(\lambda, n+m)$-Einstein metrics with $m>1 $, $\lambda \geq 0$, and $\mu \leq 0$, are the trivial ones with $\lambda = \mu = 0$.
\end{corollary}

\begin{proof}
We have just seen that if $\lambda >0$ then $\mu>0$, so we only need to consider the $\lambda = 0$ case. Corollary \ref{cormuboundary} tells us that since $\mu\leq 0$, $\partial M = \emptyset$. Therefore, we are in position to apply the main theorem in \cite{Case1} which says exactly that if $\lambda = 0$, $\partial M \neq \emptyset$, and $\mu \leq 0$, then the space is trivial.
\end{proof}

Returning to the compact case, the formula in Proposition \ref{propmuintegral} also gives us the following.

\begin{corollary}
If $M$ is compact and $m>1$, then
\begin{equation*}
\int_{M}(\mathrm{scal}-n\lambda )\df\mu_{m}=-m(m-1)\int_{M}|\nabla
w|^{2}\df\mu_{m-2}.
\end{equation*}
In particular,
\begin{equation*}
\int_{M}(\mathrm{scal}-n\lambda )\df\mu_{m}\leq 0
\end{equation*}
and is zero if and only if $w$ is constant.
\end{corollary}

\begin{proof}
From Proposition \ref{propmu}, we have
\begin{equation*}
\mu =w\Delta w+(m-1)|\nabla w|^{2}+\lambda w^{2}.
\end{equation*}
Then Proposition \ref{propmuintegral} implies that
\begin{equation*}
\int_{M}(w\Delta w+(m-1)|\nabla w|^{2})\df\mu_{m-2}=0.
\end{equation*}
We also have
\begin{equation*}
\Delta w=\frac{w}{m}(\mathrm{scal}-n\lambda ),
\end{equation*}
so
\begin{equation*}
-m(m-1)\int_{M}|\nabla w|^{2}\df\mu_{m-2} = \int_{M}(\mathrm{scal}-n\lambda
)\df\mu_{m}
\end{equation*}
which shows the desired identity.
\end{proof}

\medskip

\section{The Laplacian of the scalar curvature and applications}

In this section we review the formula from \cite{CSW} for the Laplacian of
the scalar curvature. This formula is similar to the formula for gradient
Ricci solitons and we fix notation which emphasizes this similarity and will
lead us to the formulas in the next section. In this section we also verify
that the applications of the formula from \cite{CSW} extend to the boundary
case.

First we recall the formulas for the scalar curvature of a gradient Ricci
soliton.

\begin{proposition}
Let $(M,g,f)$ be a gradient Ricci soliton
\begin{eqnarray*}
\mathrm{Ric} + \mathrm{Hess} f = \lambda g
\end{eqnarray*}
then
\begin{eqnarray*}
\frac{1}{2} \nabla \mathrm{scal} &=& \mathrm{Ric}(\nabla f) \\
\frac{1}{2} \left( \Delta - D_{\nabla f} \right)(\mathrm{scal}) &=& \lambda
\mathrm{scal} - |\mathrm{Ric}|^2
\end{eqnarray*}
\end{proposition}

For a $(\lambda, m+n)$-Einstein manifold, the scalar curvature is constant
when $m=1$. When $m>1$ we define
\begin{eqnarray*}
\rho (x)&=&\frac{1}{m-1}((n-1)\lambda -\mathrm{scal}) \\
P &=& \mathrm{Ric} - \rho g.
\end{eqnarray*}

The next proposition lists formulas (3.12) and (3.13) of \cite{CSW} in terms
of $\rho$ and $P$.

\begin{proposition}[\protect\cite{CSW}]
Let $(M,g,w)$ be a $(\lambda,n+m)$-Einstein manifold, then
\begin{eqnarray*}
\frac{w}{2}\nabla \rho &=&P(\nabla w) \\
\frac{1}{2}L_{m+1}(\mathrm{scal}) &=&(\lambda -\rho )\mathrm{tr}(P)-|P|^{2}.
\end{eqnarray*}
\end{proposition}

Applying Lemma \ref{IntByParts} this immediately gives us

\begin{corollary}
On a compact $(\lambda,n+m)$-Einstein manifold with $m>1$,
\begin{equation*}
\int_{M}\left( (\lambda -\rho )\mathrm{tr}\left( P\right) -|P|^{2}\right)
\df\mu_{m+1}=0
\end{equation*}
\end{corollary}

It is also useful to re-write the formula for the Laplacian of the scalar
curvature as
\begin{eqnarray*}
\frac{1}{2}L_{m+1}(\mathrm{scal}) &=&(\lambda -\rho )\mathrm{tr}(P)-|P|^{2}
\\
&=&\mathrm{tr}(P)\left( (\lambda -\rho )-\frac{\mathrm{tr}(P)}{n}\right)
-\left\vert P-\frac{\mathrm{tr}(P)}{n}g\right\vert ^{2} \\
&=&\mathrm{tr}(P)\left( \lambda -\frac{\mathrm{scal}}{n}\right) -\left\vert
P-\frac{\mathrm{tr}(P)}{n}g\right\vert ^{2}
\end{eqnarray*}

\begin{proposition}
On a non-trivial compact $(\lambda ,n+m)$-Einstein manifold with $m>1$,
\begin{equation*}
\mathrm{tr}(P)\geq 0.
\end{equation*}
Moreover, if $\mathrm{tr}(P)=0$ at an interior point of $M$ then the metric is $\rho $-Einstein.
\end{proposition}

\begin{proof}
For $\varepsilon > 0$, set
\begin{eqnarray*}
U_{\varepsilon } &=&\{x\in M:\mathrm{tr}(P)\leq -\varepsilon \} \\
&=&\left\{ x\in M:\mathrm{scal}\leq \frac{n(n-1)}{n+m-1}\lambda -\varepsilon
\right\}.
\end{eqnarray*}
Assume for contradiction that $U_{\varepsilon }$ has non-empty interior and that $\frac{n(n-1)}{n+m-1}\lambda -\varepsilon $ is a regular value for $\mathrm{scal}$. We have
\begin{equation*}
\int_{U_{\varepsilon}}L_{m+1}\left( \mathrm{scal}\right)
\df\mu_{m+1}=-\int_{\partial U_{\varepsilon}}g(\nabla \mathrm{scal},\eta
)w^{m+1}\df \mathrm{vol}_{g},
\end{equation*}
where $\eta$ is the unit outward-pointing normal vector field of $\partial U_{\varepsilon}$. For the right hand side there are two cases. On the one hand, at a point where $x\in \partial M$, $w^{m+1} =0$, so the integrand on the right hand side vanishes. On the other hand, when $x\in \mathrm{int}(M)$, we know that $\eta =-\frac{\nabla \mathrm{scal}}{\left\vert \nabla \mathrm{scal}\right\vert }$. In either case, we see that the right hand side is nonnegative, so
\begin{equation*}
\int_{U_{\varepsilon}}L_{m+1}\left( \mathrm{scal}\right) \df\mu_{m+1}\geq 0.
\end{equation*}

However, we also have
\begin{equation*}
\frac{1}{2}L_{m+1}(\mathrm{scal})=\mathrm{tr}(P)\left( \lambda -\frac{\mathrm{scal}}{n}\right) -\left\vert P-\frac{\mathrm{tr}(P)}{n}g\right\vert^{2}
\end{equation*}
The right hand side is nonpositive on $U_{\varepsilon}$ and negative in the interior of $U_{\varepsilon }$. Therefore, this gives us a contradiction. As the set of regular values for $\mathrm{scal}$ are dense we see that $\mathrm{tr}\left( P\right) \geq 0$ on $M.$

For the last statement, we can apply the strong minimum principle to the
interior minimum to get $\mathrm{tr}(P)=0$ everywhere in $M$. Then the
formula for the Laplacian of scalar curvature gives us
\begin{eqnarray*}
|P|^2 = 0
\end{eqnarray*}
everywhere, and so the metric is $\rho$-Einstein.
\end{proof}

\begin{remark}
The inequality $\mathrm{tr}(P)\geq 0$ is equivalent to $\mathrm{scal}\geq \frac{n(n-1)}{ n+m-1}\lambda $. This is always true when $m=1$, since $\mathrm{scal}=(n-1)\lambda $ in this case.
\end{remark}

In view of the equations we are after in the next section, we also note
another formula involving $P$.

\begin{proposition}
On a $(\lambda ,n+m)$-Einstein manifold $(M,g,w)$,
\begin{equation*}
\mathrm{div}(w^{m+1}P)=0.
\end{equation*}
\end{proposition}

\begin{proof}
This is a consequence of the formula for the gradient of the scalar
curvature and the Bianchi identity.
\begin{eqnarray*}
\mathrm{div}(w^{m+1}P) &=&w^{m+1}\mathrm{div}P+P\left( \nabla w^{m+1}\right)
\\
&=&w^{m+1}\mathrm{div}(\mathrm{Ric})-w^{m+1}\nabla \rho +\left( m+1\right)
w^{m}P\left( \nabla w\right) \\
&=&\frac{1}{2}w^{m+1}\nabla \mathrm{scal}-w^{m+1}\nabla \rho +\frac{\left(
m+1\right) }{2}w^{m+1}\nabla \rho \\
&=&-\frac{m-1}{2}w^{m+1}\nabla \rho -w^{m+1}\nabla \rho +\frac{\left(
m+1\right) }{2}w^{m+1}\nabla \rho \\
&=&0.
\end{eqnarray*}
\end{proof}

\begin{remark}
On a gradient Ricci soliton the corresponding formula is $\mathrm{div}(e^{-f}\mathrm{Ric})=0$.
\end{remark}

\medskip

\section{New formulas for $(\protect\lambda, n+m)$-Einstein manifolds with $m>1 $}

Let $1<m<\infty $ and let $(M,g,w)$ be a $(\lambda,n+m)$-Einstein metric. In
the previous section, motivated by similar formulas for gradient Ricci
solitons, we defined a $(0,2)$-tensor $P$ which satisfies the equations
\begin{eqnarray*}
\frac{w}{2}\nabla \rho &=&P(\nabla w) \\
\frac{1}{2}L_{m+1}(\mathrm{scal}) &=&(\lambda -\rho )\mathrm{tr}(P)-|P|^{2}
\\
\mathrm{div}(w^{m+1}P) &=&0.
\end{eqnarray*}

There is another useful equation on Ricci solitons involving the full
curvature tensor(see \cite{XCao}),
\begin{eqnarray*}
\mathrm{div}(e^{-f} R) = 0.
\end{eqnarray*}
In this section we prove a similar formula for $m<\infty$. That is, we
define a new algebraic curvature tensor, $Q$, which satisfies
\begin{equation}  \label{divQ}
\mathrm{div}(w^{m+1}Q)=0
\end{equation}
and also traces to a multiple of $P$. To this end let
\begin{eqnarray*}
Q &=&R+\frac{2}{m}\mathrm{Ric}\odot g-\frac{(\lambda +\rho )}{m}g\odot g \\
&=&R+\frac{2}{m}P\odot g+\frac{(\rho -\lambda )}{m}g\odot g,
\end{eqnarray*}
where, for two symmetric (0,2)-tensors $s$ and $r$, we define the
Kulkarni-Nomizu product $s\odot r$ to be the $(0,4)$-tensor
\begin{equation*}
(s\odot r)(X,Y,Z,W) = \frac{1}{2}\left( r(X,W)s(Y,Z)+r(Y,Z)s(X,W) -r(X,Z)s(Y,W)-r(Y,W)s(X,Z)\right).
\end{equation*}

The formula for $Q$ arises naturally if we consider conformal changes of the
metric. Namely if we re-write the warped metric $g_{E}$ as
\begin{eqnarray*}
g_{E} &=&g_{M}+w^{2}g_{F} \\
&=&w^{2}\left( w^{-2}g_{M}+g_{F}\right)
\end{eqnarray*}
we see that the metric $w^{-2}g_{M}+g_{F}$ is conformally Einstein. Consider
the metric on the base,
\begin{equation*}
\tilde{g}=w^{-2}g.
\end{equation*}
Which is now a metric that blows up at $\partial M$. For this metric, $Q$ is
the leading order term in the formula for the curvature tensor of $\tilde{g}$
in terms of the curvature of $g$. Namely, by applying the formulas in 1.159
in \cite{Besse} along with the equation for $\mu $ one can show that,
\begin{eqnarray*}
R_{\tilde{g}} &=& w^{-2} Q + \frac{\mu}{m-1} g \odot g.
\end{eqnarray*}

Now we turn our attention to the calculations showing that $Q$ possesses the
properties we are after. First we verify that if we trace $Q$ over $M$ we
obtain a multiple of $P$.

\begin{proposition}
Let $E_i$ be an orthonormal basis, then
\begin{eqnarray*}
\sum_{i=1}^n Q(X, E_i, E_i, Y) &=& \frac{m+n-2}{m}P(X,Y) \\
\sum_{i,j=1}^n Q(E_j, E_i, E_i, E_j) &=& \frac{m+n-2}{m(m-1)} \left( (m+n-1)
\mathrm{scal} -(n(n-1))\lambda \right).
\end{eqnarray*}
\end{proposition}

\begin{proof}
From the definition of $Q$, we have
\begin{eqnarray*}
\sum_{i=1}^{n}Q(X,E_{i},E_{i},Y) &=&\mathrm{Ric}+\frac{1}{m}\sum_{i=1}^{n}(\mathrm{Ric}(X,Y)g(E_{i},E_{i})+\mathrm{Ric}(E_{i},E_{i})g(X,Y) \\
&&-\mathrm{Ric}(X,E_{i})g(Y,E_{i})-\mathrm{Ric}(Y,E_{i})g(X,E_{i})) \\
&&-\frac{\lambda +\rho }{m}\sum_{i=1}^{n}\left(
g(X,Y)g(E_{i},E_{i})-g(X,E_{i})g(Y,E_{i})\right) \\
&=&\mathrm{Ric}(X,Y)+\frac{1}{m}\left( (n-2)\mathrm{Ric}(X,Y)+\mathrm{scal}g(X,Y)\right) -\frac{\lambda +\rho }{m}(n-1)g(X,Y) \\
&=&\frac{m+n-2}{m}\mathrm{Ric}(X,Y) +\frac{1}{m}\left( \mathrm{scal}-(n-1)\lambda -(n-1)\rho \right) g(X,Y) \\
&=&\frac{m+n-2}{m}\left( \mathrm{Ric}(X,Y)-\rho g(X,Y)\right) .
\end{eqnarray*}
Note that
\begin{equation*}
\mathrm{tr}(P)=(n-1)\lambda -(m+n-1)\rho,
\end{equation*}
which gives the second identity.
\end{proof}

The gradient Ricci soliton equation immediately implies
\begin{equation*}
\left( \nabla _{X}\mathrm{Ric}\right) \left( Y,Z\right) -\left( \nabla _{Y}\mathrm{Ric}\right) \left( X,Z\right) =R\left( X,Y,Z,\nabla f\right) .
\end{equation*}
There is a similar but slightly more complicated formula for $\left( \lambda
,m+n\right) $-Einstein manifolds involving $Q$.

\begin{proposition}
\label{propdiffP} Let $(M,g,w)$ be a $(\lambda ,n+m)$-Einstein manifold,
then
\begin{equation*}
\frac{w}{m}\left( (\nabla _{X}P)(Y,Z)-(\nabla _{Y}P)(X,Z)\right)
=-Q(X,Y,Z,\nabla w)-\frac{1}{m}(g\odot g)\left( X,Y,Z,P(\nabla w)\right) .
\end{equation*}
\end{proposition}

\begin{proof}
From the $\left( \lambda ,m+n\right) $-Einstein equation, it follows that
\begin{eqnarray*}
R\left( X,Y,\nabla w,Z\right) &=&\left( \nabla _{X}\left( \frac{w}{m}\left(
\mathrm{Ric}-\lambda g\right) \right) \right) \left( Y,Z\right) -\left(
\nabla _{Y}\left( \frac{w}{m}\left( \mathrm{Ric}-\lambda g\right) \right)
\right) \left( X,Z\right) \\
&=&\frac{w}{m}\left( (\nabla _{X}P)(Y,Z)-(\nabla _{Y}P)(X,Z)\right) \\
&&+\frac{1}{m}g\left( X,\nabla w\right) P\left( Y,Z\right) -\frac{1}{m}g\left( Y,\nabla w\right) P\left( X,Z\right) \\
&&-\frac{1}{m}g\left( X,\nabla \left( w\left( \lambda -\rho \right) \right)
\right) g\left( Y,Z\right) +\frac{1}{m}g\left( Y,\nabla \left( w\left(
\lambda -\rho \right) \right) \right) g\left( X,Z\right) \\
&=&\frac{w}{m}\left( (\nabla _{X}P)(Y,Z)-(\nabla _{Y}P)(X,Z)\right) -\frac{\lambda -\rho }{m}(g\odot g)(X,Y,Z,\nabla w) \\
&&+\frac{1}{m}g\left( X,\nabla w\right) P\left( Y,Z\right) -\frac{1}{m}g\left( Y,\nabla w\right) P\left( X,Z\right) \\
&&+\frac{1}{m}g\left( X,w\nabla \rho \right) g\left( Y,Z\right) -\frac{1}{m}g\left( Y,w\nabla \rho \right) g\left( X,Z\right) \\
&=&\frac{w}{m}\left( (\nabla _{X}P)(Y,Z)-(\nabla _{Y}P)(X,Z)\right) -\frac{\lambda -\rho }{m}(g\odot g)(X,Y,Z,\nabla w) \\
&&+\frac{2}{m}\left( P\odot g\right) (X,Y,Z,\nabla w) +\frac{1}{m}(g\odot
g)(X,Y,Z,P\left( \nabla w\right) ).
\end{eqnarray*}
Rearranging the terms then proves the identity.
\end{proof}

Finally we prove the identity involving the divergence of $Q$.

\begin{proposition}
Let $(M,g,w)$ be a $(\lambda ,n+m)$-Einstein manifold, then we have
\begin{equation*}
\mathrm{div}(w^{m+1}Q)=0.
\end{equation*}
\end{proposition}

\begin{proof}
It is convenient for the proof to contract against the fourth variable in $Q$
keeping $X,Y,Z$ in their natural places for the curvature tensor. The
statement then is equivalent to
\begin{equation*}
w\mathrm{div}Q\left( X,Y,Z\right) =-\left( m+1\right) Q\left( X,Y,Z,\nabla
w\right).
\end{equation*}

We start with
\begin{eqnarray*}
w\frac{2}{m}(\mathrm{div}(P\odot g))(X,Y,Z) &=&\frac{w}{m}(\mathrm{div} P)(X))g(Y,Z)-\frac{w}{m}(\mathrm{div}P)(Y)g(X,Z) \\
& & +\frac{w}{m}(\nabla_{X}P)(Y,Z)-\frac{w}{m}(\nabla _{Y}P)(X,Z) \\
&=&-\frac{m+1}{m}P\left( X,\nabla w\right) g\left( Y,Z\right) +\frac{m+1}{m} P\left( Y,\nabla w\right) g\left( X,Z\right) \\
&&-\left( Q(X,Y,Z,\nabla w)+\frac{1}{m}(g\odot g)\left( X,Y,Z,P(\nabla
w)\right) \right) \\
&=&-Q(X,Y,Z,\nabla w)-\frac{m+2}{m}(g\odot g)\left( X,Y,Z,P(\nabla w)\right)
\end{eqnarray*}
and
\begin{eqnarray*}
\frac{w}{m}(\mathrm{div}(\left( \rho -\lambda \right) g\odot g))(X,Y,Z) &=& \frac{w}{m}(g\odot g)(X,Y,Z,\nabla \rho ) \\
&=&\frac{2}{m}(g\odot g)(X,Y,Z,P\left( \nabla w\right) ),
\end{eqnarray*}
which gives us
\begin{equation*}
w(\mathrm{div}Q)(X,Y,Z)=w(\mathrm{div}R)(X,Y,Z)-Q(X,Y,Z,\nabla w)-(g\odot
g)\left( X,Y,Z,P(\nabla w)\right).
\end{equation*}
From Proposition \ref{propdiffP} we have
\begin{eqnarray*}
w(\mathrm{div}R)(X,Y,Z) &=&w(\nabla _{X}\mathrm{Ric})(Y,Z)-w(\nabla _{Y} \mathrm{Ric})(X,Z) \\
&=&w(\nabla _{X}P)(Y,Z)-w(\nabla _{Y}P)(X,Z) +wg\left( X,\nabla \rho \right) g\left( Y,Z\right) -wg\left( Y,\nabla \rho \right) g\left( X,Z\right) \\
&=&w(\nabla _{X}P)(Y,Z)-w(\nabla _{Y}P)(X,Z)+w(g\odot g)\left( X,Y,Z,\nabla
\rho \right) \\
&=&w(\nabla _{X}P)(Y,Z)-w(\nabla _{Y}P)(X,Z)+2(g\odot g)\left( X,Y,Z,P\left(
\nabla w\right) \right) \\
&=&-mQ(X,Y,Z,\nabla w)-(g\odot g)\left( X,Y,Z,P(\nabla w)\right) +2(g\odot g)\left( X,Y,Z,P\left( \nabla w\right) \right) \\
&=&-mQ(X,Y,Z,\nabla w)+(g\odot g)\left( X,Y,Z,P(\nabla w)\right)
\end{eqnarray*}
and hence the result follows.
\end{proof}

\medskip

\section{Proof of Theorem \protect\ref{thmharmonic-Weyl}}

In this section we apply the calculations in the previous section to prove
Theorem \ref{thmharmonic-Weyl}. First we recall some definitions.

\begin{definition}
Let $n\geq 3$ and let $(M^n,g)$ be a Riemannian manifold. The \emph{Schouten
tensor} is the $(0,2)$-tensor
\begin{eqnarray*}
S = \mathrm{Ric} - \frac{\mathrm{scal}}{2(n-1)} g.
\end{eqnarray*}
We say $(M^n, g) $ has \emph{harmonic Weyl tensor} if $S$ is a Codazzi
tensor, i.e.,
\begin{eqnarray*}
(\nabla_X S)(Y,Z) = (\nabla_Y S)(X,Z) \qquad \mbox{for any } X, Y, Z.
\end{eqnarray*}
\end{definition}

\begin{remark}
In dimension three, harmonic Weyl tensor is equivalent to $(M^{3},g)$ being
locally conformally flat. If $n>3$, the Weyl tensor is defined via the
formula
\begin{equation*}
R=W+\frac{2}{n-2}\mathrm{Ric}\odot g-\frac{\mathrm{scal}}{(n-1)(n-2)}g\odot
g,
\end{equation*}
and, as the language suggests, $\mathrm{div}(W) = 0$ if and only if $M$ has
harmonic Weyl tensor. Recall that when $n=3$, $W=0$.
\end{remark}

\begin{remark}
Another equivalent formulation is $M$ has harmonic Weyl tensor if and only
if
\begin{eqnarray*}
\mathrm{div} R(X,Y,Z) &=& \frac{1}{2(n-1)} (g \odot g)(X,Y,Z, \nabla \mathrm{scal}).
\end{eqnarray*}
\end{remark}

Now we see how harmonic Weyl tensor affects the formulas between $P$ and $Q$.

\begin{proposition}
If $(M,g,w)$ is a $(\lambda ,n+m)$-Einstein manifold and has harmonic Weyl
tensor, then
\begin{eqnarray}
Q(X,Y,Z,\nabla w) &=&\frac{m+n-2}{m(n-1)}\left( P(\nabla w,X)g(Y,Z)-P(\nabla
w,Y)g(X,Z)\right)  \notag \\
&=&\frac{m+n-2}{m(n-1)}(g\odot g)\left( X,Y,Z,P\left( \nabla w\right) \right)
\label{divWP}
\end{eqnarray}
\end{proposition}

\begin{proof}
From Proposition \ref{propdiffP}, we have
\begin{equation*}
w(\mathrm{div}R)(X,Y,Z)=-mQ(X,Y,Z,\nabla w)+(g\odot g)\left( X,Y,Z,P(\nabla
w)\right).
\end{equation*}
On the other hand we have
\begin{eqnarray*}
w(\mathrm{div}R)(X,Y,Z) &=&w\frac{1}{2(n-1)}(g\odot g)(X,Y,Z,\nabla \mathrm{scal}) \\
&=&-w\frac{m-1}{2\left( n-1\right) }(g\odot g)(X,Y,Z,\nabla \rho ) \\
&=&-\frac{m-1}{n-1}(g\odot g)(X,Y,Z,P\left( \nabla w\right) ).
\end{eqnarray*}
These two equations combine to give the desired identity.
\end{proof}

This gives us the following corollary, a similar lemma for Ricci solitons is
proven in \cite{FLGR}.

\begin{corollary}
If $(M,g,w)$ is a $(\lambda ,n+m)$-Einstein manifold and has harmonic Weyl
tensor then, at a point where $\nabla w\neq 0$, $\nabla w$ is an eigenvector
for $P$. Moreover, if $X,Y,Z\perp \nabla w$ then
\begin{eqnarray}  \label{divWQ}
Q(X,Y,Z,\nabla w) &=&0 \\
Q(\nabla w,Y,Z,\nabla w) &=&\frac{m+n-2}{m(n-1)}P(\nabla w,\nabla w)g(Y,Z).
\end{eqnarray}
\end{corollary}

\begin{proof}
To see that $\nabla w$ is an eigenvector for $P$ set $Z=\nabla w$ in (\ref{divWP}) to obtain
\begin{equation*}
P(\nabla w,X)g(\nabla w,Y)-P(\nabla w,Y)g(X,\nabla w)=0 \quad \mbox{for any }
X,Y.
\end{equation*}
Now we know that $\nabla w$ is an eigenvector and $P(X, \nabla w) = 0$ when $X \perp \nabla w$. Combining this again with (\ref{divWP}) gives the other two formulas.
\end{proof}

\begin{remark}
\label{remHess} $\nabla w$ is an eigenfield for $P$ if and only if $\nabla w$ is an eigenfield for $\mathrm{Hess}w$. This implies that $w$ is rectifiable, i.e., $|\nabla w|^{2}$ is constant on the connected components of the level sets of $w$, since, if $X\perp \nabla w$, then
\begin{equation*}
D_{X}|\nabla w|^{2}=2\mathrm{Hess}w(\nabla w,X)=0.
\end{equation*}
In particular, the connected components of the regular levels sets for $w$ form a Riemannian foliation of an open subset of $M.$
\end{remark}

Following the soliton proof in \cite{CaoChen}, we consider the Weyl tensor
in order to get control on the other eigenvalues of $P$. In the next
proposition we record the decomposition of $Q$ in terms of the Weyl tensor.

\begin{proposition}
If $(M,g,w)$ is a $(\lambda ,n+m)$-Einstein manifold and has harmonic Weyl
tensor with $m>1$, then
\begin{equation}  \label{Qdecomp}
Q=W+\frac{2(n+m-2)}{m(n-2)}(P\odot g)-\frac{n+m-2}{m(n-1)(n-2)}\mathrm{tr}\left(P\right)(g\odot g).
\end{equation}
\end{proposition}

\begin{proof}
We have
\begin{eqnarray*}
Q &=&R+\frac{2}{m}P\odot g+\frac{\rho -\lambda }{m}g\odot g \\
R &=&W+\frac{2}{n-2}P\odot g+\left( \frac{2\rho }{n-2}-\frac{\mathrm{scal}}{(n-1)(n-2)}\right) g\odot g
\end{eqnarray*}
Putting these together and using that
\begin{eqnarray*}
\mathrm{scal} &=&(n-1)\lambda -(m-1)\rho \\
\mathrm{tr}(P) &=&-(m+n-1)\rho +(n-1)\lambda
\end{eqnarray*}
gives us
\begin{eqnarray*}
Q &=&W+\frac{2(m+n-2)}{m(n-2)}(P\odot g) \\
& & +\left( \frac{((n+2m-2)(n-1)+m(m-1))\rho }{m(n-1)(n-2)}-\frac{(n+m-2)\lambda }{m(n-2)}\right) g\odot g \\
&=&W+\frac{2(m+n-2)}{m(n-2)}(P\odot g) +\frac{m+n-2}{m(n-1)(n-2)}\left((m+n-1)\rho -(n-1)\lambda \right) g\odot g \\
&=&W+\frac{2(m+n-2)}{m(n-2)}(P\odot g)-\frac{m+n-2}{m(n-1)(n-2)}\mathrm{tr}(P)g\odot g.
\end{eqnarray*}
\end{proof}

\begin{lemma}
\label{lemmaEigen} If $M$ is a $(\lambda ,n+m)$-Einstein manifold with
harmonic Weyl tensor and $W\left( \nabla w,Y,Z,\nabla w\right) =0$, then at
a point $p$ where $\nabla w\neq 0$, $P$ (or $\mathrm{Ricci}$) has at most
two eigenvalues, and if it has two eigenvalues then one has multiplicity 1
with eigenvector $\nabla w$, and the other with multiplicity $n-1$.
\end{lemma}

\begin{proof}
We already know that $\nabla w$ is an eigenvector for $P$. Let $Y,Z \perp
\nabla w$. Since $W(\nabla w,Y,Z,\nabla w)=0$ by (\ref{Qdecomp}) we have
\begin{eqnarray*}
Q(\nabla w,Y,Z,\nabla w) &=&\frac{2(n+m-2)}{m(n-2)}(P\odot g)(\nabla
w,Y,Z,\nabla w) \\
&&\qquad -\frac{n+m-2}{m(n-1)(n-2)}\mathrm{tr}(P)(g\odot g)(\nabla
w,Y,Z,\nabla w) \\
&=&\frac{(n+m-2)}{m(n-2)}\left( P(\nabla w,\nabla w)g(Y,Z)+P(Y,Z)|\nabla
w|^{2}\right) \\
&&\qquad -\frac{n+m-2}{m(n-1)(n-2)}\mathrm{tr}(P)|\nabla w|^{2}g(Y,Z).
\end{eqnarray*}
On the other hand by the identity (\ref{divWQ}) we also have
\begin{equation*}
Q(\nabla w,Y,Z,\nabla w)=\frac{m+n-2}{m(n-1)}P(\nabla w,\nabla w)g(Y,Z).
\end{equation*}
Equating these equations gives us
\begin{equation*}
(n-1)P(Y,Z)|\nabla w|^{2}=\left( \mathrm{tr}(P)|\nabla w|^{2}-P(\nabla
w,\nabla w)\right) g(Y,Z)
\end{equation*}
which implies that $Y$ and $Z$ are eigenvectors for $P$ with the same
eigenvalue.
\end{proof}

\smallskip

Now we turn our attention to finishing the proof of the theorem. We have shown that the Schouten tensor, $S$, has at most two eigenvalues when $dw\neq 0$. Let $\sigma _{1}$ and $\sigma _{2}$ be the eigenvalue functions of $S$ and define
\begin{equation*}
O=\{x\in M:dw\neq 0,\sigma _{1}(x)\neq \sigma _{2}(x)\}.
\end{equation*}
First we prove a local result about the metric around points in $O$.

\begin{theorem}
\label{ThmLocal} Let $m>1$ and let $(M,g,w)$ be a $(\lambda, n+m)$-Einstein
metric such that $(M,g)$ has harmonic Weyl tensor and $W(\nabla w,\cdot
,\cdot ,\nabla w)=0$ in an open set containing $p \in O$. Then
\begin{eqnarray*}
g &=& \mathrm{d}t^2 + \psi^2(t)g_N \\
w&=& w(t).
\end{eqnarray*}
around $p$. Where $g_N$ is an Einstein metric. If the metric is locally
conformally flat in a neighborhood of $p$, then $N$ must be a space of
constant curvtaure.
\end{theorem}

\begin{proof}
To fix notation let $\sigma_1$ be the eigenvalue of $S$ with eigenvector $ \nabla w$ and let $\sigma_2$ be the eigenvalue with eigenspace the the orthogonal complement of $\nabla w$. Since the dimension of eigenspace of $\sigma_2$ is bigger than one, 16.11(iii) in \cite{Besse} shows that $\sigma_2 $ is locally constant on the level sets of $w$ in $O$.

Using the $(\lambda, n+m)$-Einstein equation we see that $\mathrm{Hess}w$ also has at most two eigenvalues, call them $\mu_1$ and $\mu_2$ where the eigenspaces for $\mu_i$ correspond to those for $\sigma_i$ and $\mu_i$ and $\sigma_i$ are related by the formula
\begin{equation*}
\mu_i = \frac{w}{m} \left( \sigma_i + \frac{\mathrm{scal}}{2(n-1)} - \lambda
\right) \quad i=1, 2.
\end{equation*}

Now Remark \ref{remHess} shows that $\mu _{1}$ is locally constant on the
level sets of $w$ since
\begin{equation*}
\frac{1}{2}D_{\nabla w}|\nabla w|^{2} = \mathrm{Hess}w(\nabla w,\nabla w)
\end{equation*}
Moreover if $X \perp \nabla w$ then
\begin{equation*}
D_{X}\rho =\frac{2}{w}P(\nabla w,X)=0.
\end{equation*}
So $\mathrm{scal}$ is also locally constant on the level sets of $w$. This
implies that $\sigma _{1}$ is locally constant on the level sets of $w$.
Again using Remark \ref{remHess} we now have that $|\nabla w|^2$, $\mu_1$
and $\mu_2$ are all locally constant on the level sets of $w$.

Now we can write the metric in a neighborhood of $p$ as
\begin{equation*}
g=\frac{1}{|\nabla w|^{2}}\mathrm{d}w\otimes \mathrm{d}w+g_{w}
\end{equation*}
where $g_{w}$ is the metric on the level set. And we can write $\mathrm{Hess} w$ as
\begin{equation*}
\mathrm{Hess}w=\frac{\mu _{1}}{|\nabla w|^{2}}\mathrm{d}w\otimes \mathrm{d}w+\mu _{2}g_{w}
\end{equation*}
Where $\mu _{1}$ and $\mu _{2}$ are locally functions of $w$. In particular $\mathcal{L}_{\nabla w}g_{w}=\mu _{2}g_{w}$ and we see that the metric can be written as
\begin{equation*}
g=\frac{1}{|\nabla w|^{2}}\mathrm{d}w\otimes \mathrm{d}w+\psi ^{2}g_{w_{0}}
\end{equation*}
where
\begin{equation*}
\psi \left( w\right) =\exp \left( \int_{w_{0}}^{w}\mu _{2}\left( s\right)
\df s\right)
\end{equation*}

Now any metric of this form whose Ricci tensor has at most two eigenvalues
must have $g_{w_{0}}$ Einstein. Moreover, a metric in this form is
conformally flat if and only if $g_{w_{0}}$ has constant curvature.
\end{proof}

We can now obtain the global result by patching warped product pieces
together along geodesics.

\begin{theorem}
Let $m>1$ and suppose that $(M,g)$ is complete, simply connected and has harmonic Weyl tensor and $W(\nabla w,\cdot ,\cdot ,\nabla w)=0$, then $(M,g,w)$ is a non-trivial $(\lambda, n+m)$-Einstein metric if and only if it is of the form
\begin{eqnarray*}
g &=& \mathrm{d}t^2 + \psi^2(t)g_L \\
w &=& w(t),
\end{eqnarray*}
where $g_L$ is an Einstein metric. Moreover, if $\lambda\geq0$ then $(L,g_{L})$ has non-negative Ricci curvature, and if it is Ricci flat, then $\psi$ is a constant, i.e., $(M,g)$ is a Riemannian product.
\end{theorem}

\begin{proof}
It is a direct calculation to see that any metric of the form $g = \mathrm{d} t^2 + \psi^2(t)g_L$ where $g_L$ is Einstein has harmonic Weyl tensor and satisfies $W(\nabla w,\cdot ,\cdot ,\nabla w)=0$ (see 16.26(i) in \cite{Besse}).

If $dw = 0$ in an open set of $M$, then $g$ is $\lambda$-Einstein in an open set and then trivial everywhere by analyticity. Similarly, if $\sigma_1 = \sigma_2$ in an open set, then by Schur's lemma $g$ would be $\rho$-Einstein. In this case, we already know by Proposition \ref{propkappaEinstein} that any metric which is both $\rho$-Einstein and non-trivially $(\lambda, n+m)$-Einstein satisfies the conclusion of the theorem. Therefore we can assume that the set $O$ is dense in $M$.

Choose an arbitrary point $p\in O$ and let $L$ be the connected component of the level set of $w$ that contains $p$. By Remark \ref{remHess} we know that $|\nabla w|^{2}\neq 0$ is constant on $L$ and therefore $L$ is a smooth hypersurface. Moreover, note that as $\left\vert \nabla w\right\vert ^{2}$, $\sigma _{1},$ and $\sigma _{2}$ are constant on $L,$ it follows that all accumulations points for $L$ lie in $O$ and hence also in $L.$ Thus $L$ is a closed subset of $M$ and in particular properly embedded.

As $M$ is simply connected, $L$ is two-sided. Let $t$ be the signed distance to $L$. We will work on the positive side of $L$, the other side will work in exactly the same way. Set
\begin{equation*}
A_{+}=\{a\in \mathbb{R}^{+}:g=\mathrm{d}t^{2}+\psi ^{2}(t)g_{L}\text{ and }w=w(t)\text{ on }t^{-1}([0,a])\}
\end{equation*}

First we would like to show that $A_{+}\neq \emptyset$. To this end let $q$ be a point in the  connected component of $O$ that contains $L$ with $d\left( L,q\right) =\varepsilon$. Let $L_{\varepsilon }$ be the connected component of a level set of $w$ that contains $q$.  We know that $L_{\varepsilon }$ is a properly embedded hypersurface and a connected component of $t^{-1}\left( \varepsilon \right) .$ Similarily it follows that $L$ is a connected component of the set of points that have distance $\varepsilon $ to $L_{\varepsilon }.$ Now suppose that $x\in t^{-1}\left(\varepsilon \right) .$ Then there is a minimal geodesic $\gamma _{1}:\left[0,\varepsilon \right] \rightarrow M$ with $\gamma _{1}\left( 0\right) \in L$ and $\gamma _{1}\left( \varepsilon \right) =x$. However we also know that $d\left( \gamma _{1}\left( 0\right) ,L_{\varepsilon }\right) =\varepsilon $ so there must also be a minimal geodesic $\gamma _{2}:\left[ 0,\varepsilon \right] \rightarrow M$ with $\gamma _{2}\left( 0\right) =\gamma _{1}\left( 0\right) $ and $\gamma _{2}\left( \varepsilon \right) \in L_{\varepsilon }$. Note that $L_{\varepsilon }$ and $t^{-1}\left( \varepsilon \right) $ are both on the same side of $L$ so it must follow that $\dot{\gamma}_{1}\left( 0\right) =\dot{\gamma}_{2}\left( 0\right) .$ Consequently $x\in L_{\varepsilon }$. This shows that $t$ is smooth on $t^{-1}\left( 0,\varepsilon \right) .$

If we define $g_{t}=g|_{L_{t}},$ then it follows as in Theorem \ref{ThmLocal}
that $\mathcal{L}_{\nabla t}g_{t}=\mu _{2}g_{t}$ and $g=\mathrm{d}t^{2}+\psi
^{2}\left( t\right) g_{t}$ for $t\in \left( 0,\varepsilon \right) .$ This
shows that $A_{+}\neq \emptyset$.

We now need to show that $A_{+} = \mathbb{R}^{+}$. First note that if $a \in A_{+}$ and $\psi(a) = 0$ then every normal geodesic on the positive side of $L$ must intersect when $t=a$ and therefore none of the geodesics can continue minimizing the distance to $L$ past $a$. By completeness, this
implies that $t^{-1}\left( [0,a] \right) = t^{-1} \left( [0, \infty) \right)$ and so we are done. Similarly if $a \in A_{+}$ and a point $(a, l) \in \partial M$, then since the set $\{a\} \times L$ is a level set for $w$, $\{a \} \times L$ must be a component of $\partial M$. Again we have $t^{-1}\left( [0,a] \right) = t^{-1} \left( [0, \infty) \right)$, so this case is finished.

Therefore, we can assume that $\psi(a) \neq 0$ and $(\{a\} \times L) \cap
\partial M = \emptyset$ for all $a \in A_{+}$. $A_+$ is non-empty and is
clearly closed. To finish we will show that $A_+$ is also open, and
therefore must be all of $\mathbb{R}^+$. To see this let $a \in A_+$ with
\begin{eqnarray*}
g &=& \mathrm{d}t^2 + \psi^2(t) g_L \quad t \in[0,a] \\
\psi(a) &\neq& 0 \\
w &=&w(t)
\end{eqnarray*}

Let $\Sigma = t^{-1} (a) \subset M$ and then $\Sigma$ is a smooth connected hypersurface which is equidistant to $L$. Therefore, if $\bar{t}$ is the signed distance to $\Sigma$, then $t = a + \bar{t}$. We can see by continuity that the second fundamental form and normal curvature to $\Sigma$ is constant, so $\bar{t}$ is smooth in a uniform tubular neighborhood of $\Sigma$ and so we have that $t$ is smooth on $t^{-1} ( [ 0, a + \varepsilon]) $. Applying Remark \ref{remHess} also shows that $w=w(t)$ for $t \in [0, a + \varepsilon]$.

Now we can choose $x \in O \cap t^{-1} \left( [a, a+\varepsilon] \right)$
and use the same argument as above to show that $g$ can be written as a
warped product along $t$ for an open dense subset of $[a, a+{\varepsilon}]$.
By smoothness of the metric and $t$ this implies that we have a warped
product along all of $t^{-1}([0,a + \varepsilon])$ and therefore $a+
\varepsilon \in A_+$.

To see that $g_L$ has positive Ricci curvature when $\lambda \geq 0$, we use that we know $g_{L}$ is Einstein. Let's assume the Einstein constant is $\kappa$. In case $\psi$ vanishes somewhere $L$ has to be a round sphere in order for $M$ to be a manifold, thus $\kappa>0$ in this case. Next we consider the situation where $\psi$ never vanishes. In this case we can switch the manifold $L$ without changing any of the equations as long as the new manifold has the same Einstein constant.

Suppose $\kappa<0$ and switch $L$ to be a hyperbolic space of Einstein constant $\kappa$. Thus we obtain a $(0,n+m)$-Einstein metric of the form
\begin{equation*}
\mathrm{d}t^{2}+\psi^{2}g_{H}
\end{equation*}
with $w$ as the same warping function. On this metric we consider the weighted volume form $w^{m}\psi^{n-1}\mathrm{d}t\wedge\mathrm{d}\mathrm{vol}_{H}$ where $\mathrm{d}\mathrm{vol}_{H}$ is the hyperbolic volume form on $H$. On one hand we know from \cite{BakryQian} that volume growth
with respect to this volume form is a power function of degree $\leq n+m$. On the other hand this is clearly not possible since the volume growth on $H$ is exponential. Specifically, if consider the weighted volume of the set $B(p,R)\cap\left([a,b]\times H\right)$ for a fixed interval $[a,b]$, then it is approximately the same as a fixed small constant times the volume of a ball in $H$.

Next assume $\kappa=0$. We can then replace $L$ with $\mathbb{R}^{n-1}$. This means that $\lambda=0$ as $M$ is compact when $\lambda>0$. If the manifold $M$ has no boundary, then the splitting theorems of $m$-Bakry Emery tensor \cite{FangLiZhang} tell us that the metric splits, i.e., $\psi$ is a constant function. In the following we assume that $t=0$ is the boundary. So we have $w(0)=0$ and $w^{\prime}(0)\ne0$. Since $g=\mathrm{d}t^{2}+\psi^{2}g_{L}$, the second fundamental form of the $t$-level hypersurface is given by $\frac{\psi^{\prime}}{\psi}g_{L}$, and $\psi^{\prime}(0)=0$. The $(0,n+m)$-Einstein equation is equivalent to the following
\begin{eqnarray*}
m\frac{w^{\prime\prime}}{w} & = & -(n-1)\frac{\psi^{\prime\prime}}{\psi}\\
m\frac{w^{\prime}}{w}\frac{\psi^{\prime}}{\psi} & = & -\frac{\psi^{\prime\prime}}{\psi}-(n-2)\left(\frac{\psi^{\prime}}{\psi}\right)^{2}
\end{eqnarray*}
If we define $x=\psi^{\prime}/\psi$ and $y=w^{\prime}/w$ then we obtain a system
\begin{eqnarray*}
x^{\prime} & = & -(n-1)x^{2}-mxy\\
y^{\prime} & = & -y^{2}+\frac{n-1}{m}(mxy+(n-2)x^{2})
\end{eqnarray*}
Note that $x=0$ and $y=a/t$ is a solution for all $a$. Thus no solution can cross the $y$-axis. Also note that the set $y>0$ is invariant as $y^{\prime}>0$ when $y=0$. This means that both the
first and second quadrants are also invariant. The goal is to show that the positive $y$-axis is the only solution such that $x(t)\rightarrow 0$ and $y(t)\rightarrow\infty$ as $t\searrow0$. Note that $x^{\prime}$ has the opposite sign of $x$ as long as $(n-1)x+my>0$. Thus any solution with $x(t_{0})\thickapprox0$ and $y(t_{0})>0$ will move away from the $y$-axis as $t\searrow0$, i.e., flowing backwards in time. This means that it cannot approach the $y$-axis as $t\searrow0$.

When $m$ is an integer we can prove this is in a more uniform fashion. We obtain a $\lambda$-Einstein metric
\begin{equation*}
g_{E}=\mathrm{d}t^{2}+\psi^{2}(t)g_{L}+w^{2}(t)g_{F}.
\end{equation*}
Thus the metric
\begin{equation*}
\bar{g}=\mathrm{d}t^{2}+w^{2}(t)g_{F}
\end{equation*}
is $(\lambda,(m+1)+(n-1))$-Einstein, where $\kappa$ is the Einstein constant of $g_{L}$. Then applying Corollary \ref{TableUTriangular} to $\bar{g}$, $g_{L}$ must have positive Ricci curvature if $\psi$ is not a constant.
\end{proof}

\begin{remark}
\label{remsplitting} It is also easy to see that there are some further
restrictions on which warped products are possible. When $\lambda >0$, the
manifold must be compact and when $\lambda \leq 0$ the metric can not have
compact quotients, this implies that there is no example which is a warped
product over a circle. Moreover, if one has a warped product on $(-\infty,
\infty) \times L$ then the metric clearly contains a \emph{line}. The
splitting theorem for the $m$-Bakry Emery tensor \cite{FangLiZhang} then
implies that any space of this form must be a trivial product when $\lambda
= 0$.
\end{remark}

\medskip

\appendix

\section{Warped Product Einstein spaces over surfaces}

The classification of $(\lambda, 2+m) $-Einstein spaces is discussed in \cite{Besse}. In this appendix, we add some of the details to the analysis of the equations that can be found there.

In dimension two it is shown in \cite{Besse} that the equation is equivalent to
\begin{eqnarray}  \label{eqnw2ndderivative}
2 ww^{\prime\prime}+ (m-1) (w^{\prime})^2 + \lambda w^2 = \mu
\end{eqnarray}

If $m=1$, then we have $\mu=0$ and the equation (\ref{eqnw2ndderivative}) is
integrated to
\begin{equation*}
(w^{\prime})^2 + \frac{\lambda}{2} w^2 =C
\end{equation*}
where $C$ is a constant. It is easy to see that $(M^2, g)$ has constant
curvature.

If $m>1$, then we multiply the equation (\ref{eqnw2ndderivative}) by the
integrating factor $w^{\prime}w^{m-2}$ and we obtain
\begin{eqnarray}  \label{eqnw1stderivative}
(w^{\prime})^2 = \frac{\mu}{m-1} - \frac{\lambda}{m+1} w^2 + C w^{1-m},
\end{eqnarray}
where $C$ is a constant.

When $C=0$, we obtain the various constant curvature spaces. If $\partial M \neq \emptyset$ then $C = 0$, because otherwise the right hand side of (\ref{eqnw1stderivative}) blows up as $w \rightarrow 0$. Therefore, we only get the constant curvature spaces when $\partial M \neq \emptyset$. If $\partial M = \emptyset$ and $M$ is compact then Theorem 1.2 in \cite{CSW} shows that $(M^2, g)$ is a trivial $(\lambda, m+2)$-Einstein manifold.

Next we assume that $C \ne 0$ and the manifold $M$ is non-compact without
boundary, i.e., $M = \mathbb{R}^2$.

From the equation (\ref{eqnw1stderivative}), we may assume that $w^{\prime}$
is non-negative and, if it vanishes at some point, say $t=0$, then $(t,u)$
is polar coordinates and $w^{\prime\prime}(0) = 1$. In this case, solving
the equation (\ref{eqnw2ndderivative}) at $t=0$ we have
\begin{equation}  \label{eqnwquartic}
\lambda \left(w(0)\right)^2 + 2 w(0) = \mu.
\end{equation}
If $\lambda = 0$, then we have $\mu >0 $ and $w(0) = \frac{1}{2}\mu$. By considering a multiple of $w$ if necessary, we may assume that $\mu = m-1$. The equation (\ref{eqnw1stderivative}) tells us that $C = -\left(\frac{m-1}{2}\right)^{m-1}$ and furthermore it is equivalent to the following system
\begin{equation*}
\left\{
\begin{array}{rcl}
v & = & w^{\prime} \\
v^2 & = & 1 + C w^{1-m}
\end{array}
\right.
\end{equation*}
The constant solution $w(t)=\frac{1}{2} (m-1)$ is a stationary point on the $vw$-phase plane. There is a unique trajectory with $v\geq 0$ and $w> 0$ that gives a non-trivial solution, see Example 9.118(a) in \cite{Besse}.

Next we assume that $\lambda < 0$. For a number $a > 0$ we consider $\tilde{w}(s) = lw(t)$ where $s = k t$, $k^2 = -\frac{\lambda}{m+1}$ and $l = \frac{a}{w(0)}$. Then the equation (\ref{eqnw2ndderivative}) becomes
\begin{equation*}
2 k^2 l^2 \tilde{w} \tilde{w}^{\prime\prime}+ k^2 l^2 (m-1)\left(\tilde{w}^{\prime}\right)^2 + l^2\lambda\tilde{w}^2 = \mu
\end{equation*}
i.e.,
\begin{equation}  \label{eqnwtilde2nd}
2 \tilde{w}\tilde{w}^{\prime\prime}+ (m-1)\left(\tilde{w}^{\prime}\right)^2
- (m+1)\tilde{w}^2 = \tilde{\mu},
\end{equation}
where $\tilde{\mu} = \frac{\mu}{k^2 l^2}$.

Let $\tilde{w}^{\prime\prime}(0) = \frac{1}{b}$ and, then the above equation at $s=0$ shows that
\begin{equation*}
\frac{1}{b} = \frac{m+1}{2} a + \frac{\tilde{\mu}}{2 a}.
\end{equation*}
If $\tilde{\mu}\geq 0$, then any positive $a$ gives a positive $b$. If $\tilde{\mu}< 0$, then the positivity of $b$ implies that
\begin{equation*}
\frac{m+1}{2}a + \frac{\tilde{\mu}}{2a} > 0
\end{equation*}
i.e.,
\begin{equation*}
a > \sqrt{\frac{-\tilde{\mu}}{m+1}}.
\end{equation*}
We integrate the equation (\ref{eqnwtilde2nd}) once and it gives
\begin{equation}  \label{eqnwtilde1st}
\left( \tilde{w}^{\prime}\right)^2 = \frac{\tilde{\mu}}{m-1} + \tilde{w}^2 +
C \tilde{w}^{1-m}
\end{equation}
where the constant $C$ is determined by $\tilde{w}^{\prime}(0) = 0$ and $\tilde{w}(0) = a$, and we have
\begin{equation*}
C = -\left(a^{m+1} + \frac{\tilde{\mu}}{m-1}a^{m-1}\right).
\end{equation*}
For fixed values of $a$ and $\tilde{\mu}$ there is a unique trajectory on the $\tilde{w}^{\prime}\tilde{w}$-plane with $\tilde{w}^{\prime}\geq 0$, $\tilde{w}> 0$ and $\tilde{w}(0) = a$ that gives us a unique $(-(m+1), 2+m)$-Einstein metric on $\mathbb{R}^2$, see Example 9.118(d) in \cite{Besse}.
\smallskip

Now we assume that $w^{\prime}$ is positive everywhere and then $(t,u)$ is Cartesian coordinates. By scaling $t$ and $w$ we may assume that $\mu = -(m-1), 0$, or $m-1$ and $\lambda = -(m+1)$ or $0$ in the equation (\ref{eqnw1stderivative}). Let $(a, b)$ be the range of $w$ with $w^{\prime}>0$
where $b > a > 0$ and $b$ may equal to $\infty$. The metric $g$ is complete if and only if the following two integrals diverge for a $w_0 \in (a,b)$
\begin{eqnarray*}
\int_a^{w_0} \frac{\mathrm{d} w}{\sqrt{\frac{\mu}{m-1}- \frac{\lambda}{m+1} w^2 + C w^{1-m}}} & = & \infty \\
\int_{w_0}^b \frac{\mathrm{d} w}{\sqrt{\frac{\mu}{m-1}- \frac{\lambda}{m+1} w^2 + C w^{1-m}}} & = & \infty.
\end{eqnarray*}

We consider the case when $\mu = 0$ and $\lambda = -(m+1)$. Then we have $(w^{\prime})^2 = w^2 + C w^{1-m}$. If $C > 0$, then the range of $w$ is $(0,\infty)$ and the integral from $0$ to any $w_0 >0$ converges. If $C< 0$, then the range of $w$ is $(a, \infty)$ with $a=(-C)^{\frac{1}{m+1}}$.
However the integral from $a$ to $w_0$ for any $w_0 > a$ converges. So the completeness of the metric implies that $C = 0$ and then $w = e^{t}$, see Example 9.118(b) in \cite{Besse}. Note that the warped product metric on $\mathbb{R}^2 \times F$ has Ricci curvature $-(m+1)$. The other cases follow similarly and they give the Example 9.118(c) in \cite{Besse}.

\medskip


\end{document}